\documentclass{amsart}
\usepackage{hyperref}
\usepackage{amssymb}
\usepackage{graphicx}
\usepackage[all]{xy}

\newtheorem{thm}{Theorem}[section]
\newtheorem{lem}[thm]{Lemma}
\newtheorem{cor}[thm]{Corollary}
\newtheorem{prop}[thm]{Proposition}

\theoremstyle{definition}
\newtheorem{rem}[thm]{Remark}
\newtheorem{defn}[thm]{Definition}

\theoremstyle{remark}

\renewcommand{\arraystretch}{1.2}

\def\GL{\mathrm{GL}}

\def\E{\mathcal{E}}
\def\O{\mathcal{O}}

\def\P{\mathbb{P}^1}
\def\Q{\mathbf{Q}}
\def\Z{\mathbf{Z}}
\def\C{\mathbf{C}}

\def\F{\mathbf{F}}

\def\red{\mathrm{red}}
\def\genus{\mathrm{genus}}
\def\Gal{\mathrm{Gal}}

\def\Aut{\mathrm{Aut}}
\def\det{\mathrm{det}}
\def\ord{\mathrm{ord}}

\def\trace{\mathrm{trace}}

\def\GL2{\mathrm{GL}_2}

\def\PGL2{\mathrm{PGL}_2}
\def\End{\mathrm{End}}
\def\loc{\mathrm{loc}}
\def\tors{\mathrm{tors}}
\def\rk{\mathrm{rank}}
\def\genus{\mathrm{genus}}
\def\Pic{\mathrm{Pic}}
\def\Div{\mathrm{Div}}
\def\Jac{\mathrm{Jac}}
\def\im{\mathrm{image}}

\def\isom{\xrightarrow{\sim}}
\def\map#1{\;\xrightarrow{#1}\;}

\def\ksep{k^{\mathrm{s}}}
\def\hookto{\hookrightarrow}
\def\Deltamin{\Delta_{\mathrm{min}}}
\def\bmu{\boldsymbol{\mu}}
\def\ch{\mathrm{char}}
\def\too{\longrightarrow}
\newcommand{\BP}{{\mathbb P}}
\def\Pic{\mathrm{Pic}}

\title
{On elliptic curves with an isogeny of degree 7}
\author{R. Greenberg}
\address{Department of Mathematics, University of Washington, Seattle, WA 98195, USA}
\email{greenber@math.washington.edu}
\author{K. Rubin}
\address{Department of Mathematics, UC Irvine, Irvine, CA 92697, USA}
\email{krubin@math.uci.edu}
\author{A. Silverberg}
\address{Department of Mathematics, UC Irvine, Irvine, CA 92697, USA}
\email{asilverb@math.uci.edu}
\author{M. Stoll}
\address{Mathematisches Institut,
         Universit\"at Bayreuth,
         95440 Bayreuth, Germany}
\email{Michael.Stoll@uni-bayreuth.de}

\thanks{This material is based upon work supported by the 
National Science Foundation under grants DMS-0200785, DMS-0757807, and CNS-0831004.}

\begin{document}

\begin{abstract}  
We show that if $E$ is an elliptic curve over $\Q$ with a $\Q$-rational isogeny of degree $7$, 
then the image of 
the $7$-adic Galois representation attached to $E$ is as large as allowed by the 
isogeny, except for the curves with complex multiplication by $\Q(\sqrt{-7})$.

The analogous result with $7$ replaced by a prime $p > 7$ 
was proved by the first author in \cite{greenberg}.
The present case $p = 7$ has additional interesting complications.  
We show that any exceptions correspond to the rational points on 
a certain curve of genus $12$.   We then use the method of Chabauty
to show that the exceptions are exactly the curves with complex multiplication. 

As a by-product of one of the key steps in our proof, 
we determine exactly when there exist elliptic curves over an arbitrary field $k$ 
of characteristic not $7$ 
with a $k$-rational 
isogeny of degree $7$ and a specified Galois action on the kernel of the isogeny,  
and we give a parametric description of such curves.  
\end{abstract}

\maketitle


\section{Introduction}  
\label{introsect}

Suppose that $E$ is an elliptic curve defined over $\Q$,   $p$ is a rational prime, 
$T_p(E)$ is the $p$-adic Tate module of $E$, and $G_\Q = \Gal(\bar{\Q}/\Q)$.
Then there is a natural homomorphism  
$$
\rho_{E,p}:  G_{\Q} \too \Aut_{\Z_p}(T_p(E))
$$
giving the action of $G_{\Q}$ on $T_p(E)$.  
Since $T_p(E)$ is a free $\Z_p$-module of rank $2$, $\Aut_{\Z_p}(T_p(E))$ can be identified 
(non-canonically) with $\GL2(\Z_p)$.
Serre \cite{serre} showed that if $E$ does not have complex multiplication (CM), then
$\rho_{E,p}(G_{\Q})$ has finite index in $\Aut(T_p(E))$, and
$\rho_{E,p}$ is surjective for all but finitely many primes $p$.

Suppose now that $E$ has an isogeny of degree $p$ that is defined over $\Q$ 
(in other words, $E$ has a $\Q$-rational $p$-isogeny). 
Mazur \cite{mazur} showed that 
$p$ is then in the finite set $\{ 2,3,5,7,11,13,17,19,37,43,67,163 \}$,
and if further $E$ is non-CM, then $p \in \{ 2,3,5,7,11,13,17,37 \}$.
The kernel of the isogeny is a $\Q$-rational subgroup $\Psi$ of 
$E(\bar{\Q})$ of order $p$.   
Since $\Aut(\Psi) \cong \F_p^\times$, 
the action of $G_{\Q}$ on $\Psi$ is given by a homomorphism 
$\psi: G_{\Q} \to \F^{\times}_p$. 
We refer to $\psi$ as the {\em character} of the isogeny.  

The isogeny and the corresponding character $\psi$ put 
an obvious constraint on the image of the map $\rho_{E,p}$. 
In particular, $\rho_{E,p}$ cannot be surjective.
If $E$ has complex multiplication, 
then the additional endomorphisms of $E$ put  another constraint on the image of $\rho_{E,p}$.  In that case, $\rho_{E,p}(G_\Q)$ is a $p$-adic Lie group of dimension 
$2$.  
We wish to understand whether these are the only constraints, or equivalently, whether there are 
any non-CM elliptic curves over $\Q$ for which $\rho_{E,p}(G_\Q)$ does not contain a 
Sylow pro-$p$ subgroup of $\Aut_{\Z_p}(T_p(E))$.  This is the motivation for the following 
definition.

\begin{defn}
\label{excdef}
We will say that a curve $E$ over $\Q$ is {\em $p$-exceptional} if $E$ has an isogeny of degree $p$ 
defined over $\Q$  and the image of $\rho_{E,p}$ does {\em not} contain a Sylow pro-$p$ subgroup of 
$\Aut_{\Z_p}(T_p(E))$.
\end{defn}

In other words, a curve $E$ with a $\Q$-rational $p$-isogeny
is $p$-exceptional if  the index of  $\rho_{E,p}(G_\Q)$ in $\Aut_{\Z_p}(T_p(E))$ is divisible by $p$.  
If $E$ is not $p$-exceptional, 
then the image of $\rho_{E,p}$ is as large as it could be, given the existence of a 
$\Q$-isogeny for $E$ of degree $p$ with character $\psi$. 
 Note that 
if $E$ is $p$-exceptional, then 
so is any curve $\Q$-isogenous to $E$, and if $p>2$ then 
so is any quadratic twist of $E$ 
 (see \cite{greenberg}).
If $E$ has CM, then as remarked above, $E$ is $p$-exceptional.

If $p < 7$,  then non-CM $p$-exceptional curves exist in abundance. Concerning the case $p=5$, Theorem 2 in \cite{greenberg} describes completely the possible images of $\rho_{E,5}$.   Its index  cannot be divisible by $5^2$ and is divisible by $5$ if and only if $E$ has a cyclic $\Q$-isogeny of degree $5^2$ or two independent $\Q$-isogenies of degree $5$.

In this paper we prove that the only $7$-exceptional elliptic curves
are the elliptic curves with complex multiplication by $\Q(\sqrt{-7})$.
The method of proof is as follows.
We begin with two results from \cite{greenberg}.
Let $\omega : G_\Q \to \F_p^\times$ denote the cyclotomic character giving the 
action of $G_\Q$ on  $\bmu_p$.  

\begin{thm}[\cite{greenberg}, Theorem 1]
\label{pthm}
Suppose $p \ge 7$ and $E$ is an elliptic curve over $\Q$ with a $\Q$-isogeny of degree  $p$. Let 
$\psi$ denote the character of the isogeny.  If $\psi^4 \ne \omega^{2}$, then 
$E$ is not $p$-exceptional.
\end{thm}

\begin{prop}[\cite{greenberg}, Remark 4.2.1]
\label{pprop}
Suppose $p > 7$ and $E$ is an elliptic curve over $\Q$ with a $\Q$-isogeny of degree $p$ 
and character $\psi$.   If $\psi^4 = \omega^{2}$, then $E$ has CM.
\end{prop}

These two results combine to show that there are no non-CM $p$-exceptional curves when $p > 7$.  
However,  Proposition \ref{pprop} fails when $p = 7$ (as can be seen by considering 
the family of elliptic curves with a $\Q$-isogeny of degree $7$ and character $\psi = \omega^5$; 
see \S\ref{speccasesect}).

Suppose $E$ is a $7$-exceptional elliptic curve.
By Theorem \ref{pthm}, if $\psi$ is the character of the isogeny,  then $\psi^4 = \omega^2$.  
It follows that the $\F_7^{\times}$-valued character $\psi \omega^{-5}$ has order dividing 2. Thus, replacing $E$ by a quadratic twist if necessary, we may assume 
that $\psi = \omega^5$.  
In \S\ref{rhosect} we show that if $E$ is $7$-exceptional, 
then the ratio of the minimal discriminants of $E$ and its $7$-isogenous curve is of the form $7^sw^7$
with $w\in\Q$ and $s\in\Z$. 

An explicit description of the family
$\{B_v\}$ of elliptic curves over $\Q$ with a 
$\Q$-rational $7$-isogeny with character $\psi = \omega^5$
follows from the results in \S\ref{twistsect}.
The curve $E$ is isomorphic over $\Q$ to $B_v$ for some $v \in \Q$. 
In Corollary \ref{dmbb} we show that for $v\in\Q$, the ratio of the 
minimal discriminants of $B_v$ and its $7$-isogenous curve is 
$7^{\pm 6}g(v)^6$, where $g(v):=(v^3-2v^2-v+1)/(v^3-v^2-2v+1)$.
Thus the exceptional $E$'s correspond to $\Q$-rational points
on the genus 12 curves $C_j : w^7 = 7^jg(v)$.
If $7\nmid j$,  then it turns out that  $C_j$ has no rational points over $\Q_7$, and hence  none over $\Q$.  Thus, the question is reduced to finding
the $\Q$-rational points on the curve 
$$
C_0: w^7 = (v^3-2v^2-v+1)/(v^3-v^2-2v+1).
$$
In \S\ref{finalpfsect} we use the method of Chabauty to show that 
$$C_0(\Q)= \{(0,1), (1,1), (\infty,1), (2,-1), (1/2,-1), (-1,-1)\},$$
and it follows that the only $7$-exceptional elliptic curves $E$ are the
curves with $j(E) =-15^3$ or $255^3$, i.e., the curves with complex 
multiplication by $\Q(\sqrt{-7})$.

Now suppose that  $k$ is a field of characteristic different from $7$,  
that $E$ is an elliptic curve defined over $k$, and that $E$ has a 
$k$-rational isogeny of degree $7$. 
Then the kernel of the isogeny is a $k$-rational subgroup $\Psi$ of 
$E(\ksep)$ of order $7$, where $\ksep$ denotes a fixed separable closure of $k$.   
Let $G_k = \Gal(\ksep/k)$.  
Since  $\Aut(\Psi) \cong \F_7^\times$, 
the action of $G_k$ on $\Psi$ is given by a homomorphism 
$\psi:  G_k \to \F_7^\times$. Again, we refer to $\psi$ as the {\em character} of the isogeny.  
For example, the character $\psi$ is trivial if and only if $\Psi \subset E(k)$.

Now let $\psi: G_k \to \F_7^\times$ be a fixed homomorphism.   
In \S\ref{twistsect} we describe all elliptic curves defined over $k$ 
that have a $k$-rational $7$-isogeny with character $\psi$.  
We will give explicit formulas for a family of elliptic curves $\{A_v\}$, where $v$ varies over 
an explicit Zariski open subset of the projective line $\P$, such that
\begin{itemize}
\item
for every $v\in k$, the elliptic curve $A_v$ has a $k$-rational $7$-isogeny and character $\psi$, 
\item
if $E$ is an elliptic curve over $k$ with a $k$-rational $7$-isogeny and 
character $\psi$, then $E$ is isomorphic over $k$ to $A_v$ for some $v\in k$.
\end{itemize}
See Theorems \ref{genthm} and \ref{moregenthm} for more precise statements.  
One consequence of this explicit description is that if $k \ne \F_2$, 
and if $\psi$ is any character, then there is an elliptic 
curve over $k$ that has a $k$-rational $7$-isogeny with character $\psi$ (see Corollary \ref{existence}).
Note that these explicit families are of a different nature and are constructed
differently from the explicit families of elliptic curves with a given
mod $N$ representation that were constructed earlier by the 
second and third authors of this paper.

The route we took to exactly determine $C_0(\Q)$
is interesting in itself. (See Remark \ref{storyrem} and the appendix for more information.)
By Faltings' proof of the Mordell Conjecture, $|C_0(\Q)|$ is finite.
An action of the group $S_3$ on $C_0$ shows that $|C_0(\Q)|$ is 
divisible by $6$.
A descent argument (see \S\ref{ranksect}) 
shows that the rank of the
Jacobian $J$ of $C_0$ is at most $6$.  In fact, the 6 known points generate
a subgroup of $J(\Q)$ of rank 4  and  $J$ is $\Q$-isogenous
to the square of the Jacobian of a genus 6 curve $D$ defined over $\Q$.  Thus the 
rank of $J(\Q)$ must be either $4$ or $6$.

A Chabauty argument at the prime $2$ 
then gives that $|C_0(\Q)|$ is either $6$ or $12$.
At our request, Kiran Kedlaya and Jennifer Balakrishnan, 
with help from Joseph Wetherell, computed $2$-adic Coleman integrals
that we hoped would directly rule out the case $|C_0(\Q)|=12$.
However, the computed dimensions worked out in such a way that one would
need to show that the rank of $J(\Q)$ were $4$ to use this method to
show $|C_0(\Q)|\neq 12$.

This gave motivation to determine the rank of $J(\Q)$, i.e., to determine
the parity of the rank of the Jacobian $\Jac(D)$ of the genus 6 curve $D$. 
Using data and information provided by Balakrishnan,
Sutherland, Kedlaya and the fourth author of this paper, 
Michael Rubinstein performed computations that gave convincing evidence that
the analytic rank of ${\Jac(D)}$ is $3$, so one expects $\rk(J(\Q))=6$.

This gave motivation to find additional generators of $J(\Q)$, which is
done at the end of \S\ref{ranksect} below. 
We used these additional points to finish the proof that $|C_0(\Q)|=6$, using a 
Chabauty argument at the prime $5$ (see \S\ref{finalpfsect}).

We thank Bjorn Poonen for helpful conversations.
We thank Jennifer Balakrishnan, Kiran Kedlaya, Michael Rubinstein, 
Andrew Sutherland, and Joseph Wetherell for 
computations that pointed the way to the correct path to take to
achieve Theorem \ref{final}.
We made use of PARI/GP \cite{pari} and Magma \cite{MAGMA}.
The fourth author thanks Wojciech Gajda and the University of
Pozna\'n for an invitation to spend a week there, where 
some of the work on this paper was done.

\section{The image of $\rho_{E,p}$}
\label{rhosect}

We assume throughout this section that $E$ is an elliptic curve defined over 
$\Q$ that has a $\Q$-isogeny of prime degree $p \ge 7$.  Let $\Psi$ denote the 
kernel of the isogeny and let $\Phi = E[p]/\Psi$.     The actions of 
$G_{\Q}$ on 
$\Psi$ and $\Phi$ are given by characters 
$\psi,  \varphi : G_{\Q} \to \F_p^{\times}$, respectively.  
Let $\omega : G_{\Q} \to \F_p^\times$ 
denote the cyclotomic character giving the action of $G_\Q$ on $\bmu_p$. 
That is,  if $\sigma \in G_\Q$ and  $\zeta_p$ is a primitive $p$-th root of unity in $\bar\Q$, then  
$\zeta_p^{\sigma} = \zeta_p^{\omega(\sigma)}$.  
Since 
$\psi \varphi = \omega$, which is an odd character, we have $\psi \neq \varphi$.   
Hence $\Psi \not \cong \Phi$ as $G_{\Q}$-modules.  
 
Let  $K_{\infty} = \Q(E[p^{\infty}])$ and let 
$\rho_{E,p}: G_{\Q} \to \Aut_{\Z_p}\big(T_p(E)\big)$ be the 
homomorphism giving the action of $G_{\Q}$ on the 
Tate module $T_p(E)$.  Then $\rho_{E,p}$ factors through the Galois group 
$G := \Gal(K_{\infty}/\Q)$, and defines an injective homomorphism from $G$ 
into 
$\Aut_{\Z_p}\big(T_p(E)\big)$.  To simplify the discussion, we identify $G$ 
with its image in $\Aut_{\Z_p}\big(T_p(E)\big)$. 

Recall the definition of $p$-exceptional from \S\ref{introsect}.
We would like to know whether $\rho_{E,p}(G_\Q)$ contains a Sylow pro-$p$ subgroup of 
$\Aut(T_p(E))$, or equivalently, whether the index 
$[\Aut_{\Z_p}(T_p(E)):  \rho_{E,p}(G_\Q)]$ is prime to $p$.

If we choose a basis for $T_p(E)$ to identify $\Aut(T_p(E))$ with $\GL2(\Z_p)$, 
then the Sylow pro-$p$ subgroups of $\Aut(T_p(E))$ are identified with the conjugates of 
$$
\begin{pmatrix} 1+p\Z_p & \Z_p\\ p\Z_p & 1+p\Z_p \end{pmatrix}.
$$
There are $p+1$ such conjugates, all containing $I_2 + pM_2(\Z_p)$. 

Let $K = \Q(\Psi, \Phi)$, the fixed field for the intersections of the kernels of 
$\psi$ and $\varphi$.  Then $K$ is an abelian extension of $\Q$ and $[\Q(E[p]):K]$ 
is $1$ or $p$. Since $[K:\Q]$ divides $(p-1)^2$,   it is not divisible by $p$. 
Let 
$$
S := \Gal(K_{\infty}/K).
$$ 
Then $S$ is a normal subgroup of $G$ and is the (unique) Sylow pro-$p$ subgroup of $G$.  

Let 
$$
E' := E/\Psi.
$$ 
Thus $E'$ has a $\Q$-isogeny of degree $p$ with kernel $\Phi$ and character $\varphi$.
  
\begin{rem}
\label{indisogrmk}
The assumption that $p \ge 7$ implies that an elliptic curve over $\Q$ cannot have a 
$G_{\Q}$-invariant cyclic subgroup of order $p^2$. This is due to Mazur \cite{mazur} 
for most primes, Ligozat \cite{ligozat} or Kenku \cite{kenku}   for $p=7$, and Kenku 
\cite{kenku13} for $p=13$.   
It follows that an elliptic curve over $\Q$  cannot have two independent 
$\Q$-isogenies of degree $p \ge 7$.
To see this, suppose to the contrary that $E[p] \cong \Psi \times \Phi$.  
Let $C= \{ P\in E : pP\in\Phi\} \subset E[p^2]$, which is obviously $G_{\Q}$-invariant.   
Then $C/\Psi$ is a $G_{\Q}$-invariant 
cyclic subgroup of $E'$ of order $p^2$, which is not possible.
It follows that both of the fields $\Q(E[p])$ and $\Q(E'[p])$ are cyclic extensions of 
$K$ of degree $p$.  
\end{rem}

\begin{prop}[\cite{greenberg}, Proposition  4.3.2]
\label{criterion}   
The curve $E$ is $p$-exceptional if and only if $\Q(E[p]) = \Q(E'[p])$. 
\end{prop}

The proof of this proposition in \cite{greenberg} is based on the Burnside Basis Theorem.  
The Frattini quotient of a Sylow pro-$p$ subgroup $S_p$ of $\Aut(T_p(E))$ containing $S$ 
has $\F_p$-dimension $3$.  It turns out that the image of $S$ in that Frattini 
quotient has $\F_p$-dimension $2$ if $\Q(E[p]) = \Q(E'[p])$, and $\F_p$-dimension 
$3$ if those two fields are distinct.  In the latter case, one can find a set 
of topological generators for $S_p$ in $S$, which then implies that $S = S_p$.

The following lemma will provide one way to verify that $\Q(E[p]) \neq \Q(E'[p])$.  
Note that if $L$ is a Galois extension of $\Q$ containing $K$ and 
$[L:K] = p$, then  the ramification degree for a prime $\ell$ in the extension  
$L/\Q$ is divisible by $p$ if and only if the primes of $K$ lying over $\ell$ 
are ramified in  $L/K$.  We then simply say that $\ell$ is 
ramified in  $L/K$.  
Interestingly, if $\ell \neq p$, then  $\ell$ can be ramified in at most one 
of the extensions  $\Q(E[p])/K$ or  $\Q(E'[p])/K$.  

\begin{lem}
\label{ellram}  
Assume that $\ell$ is a prime and that $\ell \neq p$.  Then 
the ramification degree of $\ell$ in at least one of the two extensions  
$\Q(E[p])/\Q$ and   $\Q(E'[p])/\Q$ is prime to $p$.   
\end{lem}

\begin{proof} Assume that the ramification degree of $\ell$ in $\Q(E[p])/\Q$
is divisible by $p$.
This implies that $E$ has bad reduction at $\ell$.  If $E$ had potentially good 
reduction at $\ell$, then the only primes that could divide the ramification degree for $\ell$ in $\Q(E[p])/\Q$ 
are $2$ and $3$ (see for example the proof of Corollary 2(a) to Theorem 2
of \cite{serretate}). This contradicts the assumption that 
$p \ge 7$.  Hence, $E$ must have multiplicative or potentially multiplicative 
reduction at $\ell$.  It follows from Proposition 23(b) of \cite{serre}
that $E$ has multiplicative reduction over $K$ at all primes above $\ell$.

Fix a prime $\lambda$ of $K_{\infty}$ lying above $\ell$, 
and let $I$ be the inertia group for $\lambda$ in $S$.
The Tate parametrization shows that  
for every $n$, the group 
$E[p^n]^I$ contains a cyclic subgroup of order $p^n$. 
Since $I$ fixes $K = \Q(\Psi, \Phi)$, we have
$\Psi \subseteq E[p]^I$.  
 On the other hand, 
since the ramification degree of $\ell$ in $\Q(E[p])/\Q$ is divisible by $p$, 
$I$ acts nontrivially on $E[p]$, and so we have $E[p]^I = \Psi$.    
Hence $E[p^n]^I$ is cyclic of order $p^n$ for every $n$.  
In particular,  multiplication by $p$ gives an $I$-equivariant isomorphism 
$E[p^2]^I/\Psi \isom \Psi$.  Therefore, we have $I$-equivariant isomorphisms
$$
E'[p] = (E/\Psi)[p] \cong E[p^2]^I/\Psi \times E[p]/\Psi \cong \Psi \times \Phi.$$
Since $I$ 
acts trivially on both $\Phi$ and $\Psi$, it acts trivially on $E'[p]$, so
$\Q(E'[p])/K$ is unramified above $\ell$.
Since $[K:\Q]$ is prime to $p$, 
it follows that the ramification degree of $\ell$ in $\Q(E'[p])/\Q$ is prime to $p$.
\end{proof}

\begin{rem}
\label{TatePeriod}
Lemma \ref{ellram} can also be proved by studying how the Tate period for $E$ 
over $\Q_{\ell}$ changes under the isogeny $E \to E'$.   The advantage of 
the above proof is that it also could be applied to the $p$-adic 
representations attached to modular forms, under suitable assumptions.  
\end{rem}
  
\begin{lem}
\label{ratio}
If $\psi\varphi^{-1}$ has order $2$, then $p \equiv 3 \pmod{4}$ and 
$\psi\varphi^{-1} = \omega^{(p-1)/2}$.
\end{lem}

\begin{proof}
Since $\psi\varphi = \omega$ and $\varphi$ has order dividing $p-1$, we have
$$
(\psi\varphi^{-1})^{\frac{p-1}{2}} = (\omega\varphi^{-2})^{\frac{p-1}{2}} 
   = \omega^{\frac{p-1}{2}}   
$$
Since $\omega$ has order $p-1$, we see that $(\psi\varphi^{-1})^{(p-1)/2}$ is nontrivial.  
If $\psi\varphi^{-1}$ is quadratic, we conclude that $(p-1)/2$ is odd and 
hence that  $(\psi\varphi^{-1})^{(p-1)/2} = \psi\varphi^{-1}$. The lemma follows.
\end{proof}

\begin{prop}  
\label{ramcriterion}  
Suppose that $\psi\varphi^{-1}$ has order $2$.  Then 
$E$ is $p$-exceptional if and only if for every prime $\ell \neq p$, the ramification 
degrees of $\ell$ in $\Q(E[p])/\Q$ and $\Q(E'[p])/\Q$ are both prime to $p$. 
\end{prop}

\begin{proof}   
Let $L = \Q(E[p])$ and  $L' = \Q(E[p])$.  By Remark \ref{indisogrmk}, $L$ and $L'$ 
are cyclic extensions of $K$ of degree $p$.  

Suppose first that $E$ is $p$-exceptional, and $\ell \ne p$.  
By Lemma \ref{ellram}, the ramification degree of $\ell$ in at least one of 
$L/\Q$ and $L'/\Q$ is prime to $p$.
But by Proposition \ref{criterion} we have $L = L'$, so 
the ramification degrees of $\ell$ in $L/\Q$ and $L'/\Q$ must both be prime to $p$.

Now suppose that for every prime $\ell \neq p$, the ramification 
degrees of $\ell$ in $L/\Q$ and $L'/\Q$ are both prime to $p$.
Let $\xi = \psi\varphi^{-1}$.
Since $\xi = \xi^{-1}$,  the action of $\Gal(K/\Q)$ on both 
$\Gal(L/K)$ and $\Gal(L'/K)$  is given by $\xi$. Let $F$ denote the 
quadratic extension of $\Q$ corresponding to $\xi$.  Then $F \subset K$, and 
$F = \Q(\sqrt{-p})$ by Lemma \ref{ratio}.  We can regard $\xi$ as a character of $\Gal(F/\Q)$.   
Since $\Gal(K/F)$ acts trivially on $\Gal(L/K)$ and $\Gal(L'/K)$, it follows that 
$L$ and $L'$ are abelian extensions of $F$. Since $[K:F]$ is prime to $p$, 
there exist cyclic extensions $J$ and $J'$ of $F$ of degree $p$ such that $L = KJ$ and  
$L' = KJ'$.  Now $\Gal(F/\Q)$ acts on both 
$\Gal(J/F)$ and $\Gal(J'/F)$ by the character $\xi$, so $J$ and $J'$ are dihedral 
extensions of $\Q$ of degree $2p$.

By our assumption on the ramification of primes $\ell \ne p$, 
the extensions $J/F$ and $J'/F$ can ramify only at primes above $p$.
The class number of $F = \Q(\sqrt{-p})$ is not divisible by $p$ 
(because it is less than $p$; see for example \cite[page 365]{dirichlet}).  
Hence, by class field theory, one sees that $F$ has only one cyclic extension of 
degree $p$ that is both unramified outside of $p$ and dihedral over $\Q$. 
(This extension is the first layer of the so-called ``anticyclotomic" $\Z_p$-extension 
of $F$.)  Therefore, we must have $J = J'$, and  hence $L = L'$.  
Now $E$ is $p$-exceptional by Proposition \ref{criterion}.  
\end{proof} 

Let $\Deltamin(E)$ and $\Deltamin(E')$ denote the discriminants of minimal integral 
models for $E$ and $E'$, respectively.  

\begin{thm}
\label{Deltacriterion}
Assume that $\psi\varphi^{-1}$ has order $2$ and that $E$ has semistable reduction 
at all primes $\ell$ dividing the conductor of $E$, except possibly $\ell = p$.  
Then $E$ is $p$-exceptional if and only if $\Deltamin(E')/\Deltamin(E) = p^a w^p$  
for some $a \in \Z$ and $w \in \Q^{\times}$.   
\end{thm}

\begin{proof} 
Suppose first that $\ell \ne p$ is a prime where $E$ has split 
multiplicative reduction. Then $E$ is a Tate 
curve over $\Q_{\ell}$.  Let $q_{E,\ell}$ denote the corresponding Tate 
period for $E$.  Then we have 
$$
\Q_{\ell}(E[p]) =  \Q_{\ell}\big(\bmu_p,  \sqrt[p]{q_{E,\ell}}\big)
$$
and therefore the ramification degree for $\ell$ in $\Q(E[p])$ is divisible 
by $p$ if and only if $\ord_{\ell}\big(q_{E,\ell}\big) \not \equiv 0 \pmod{p}$.  
Furthermore, we have (Proposition VII.5.1(b) of \cite{silverman})   
$$
\ord_{\ell}\big(\Deltamin(E)\big) = - \ord_{\ell}\big(j(E)\big) = \ord_{\ell}\big(q_{E,\ell}\big).
$$
Thus, the ramification degree for $\ell$ in  $\Q(E[p])/\Q$ is divisible by $p$ 
if and only if $\ord_{\ell}\big(\Deltamin(E)\big)$ is not divisible by 
$p$.  
This criterion is also valid if $E$ has nonsplit multiplicative reduction 
at $\ell$, since both the ramification degree for $\ell$ in $\Q(E[p])$ and 
the power of $\ell$ dividing $\Deltamin(E)$ are unchanged by twisting $E$ by 
a quadratic character that is unramified at $\ell$.  

By Lemma \ref{ellram},  at least  one of the integers $\ord_{\ell}\big(\Deltamin(E)\big)$,  
$\ord_{\ell}\big(\Deltamin(E')\big)$ is divisible by $p$.  Therefore, both are divisible 
by $p$ if and only if their difference is divisible by $p$.   Now apply 
Proposition \ref{ramcriterion}.   
\end{proof}

\section{Twisting $X_1(7)$ by characters}
\label{twistsect}

Fix a field $k$ of characteristic different from $7$. 
Suppose $\psi : G_k \to \F_7^\times$ is a homomorphism.
In this section we will construct the family 
of all elliptic curves over $k$ with a $k$-rational subgroup of order $7$ on which 
$G_k$ acts via the character $\psi$. 
The method of our construction is as follows.  When $\psi = 1$, we are parametrizing 
elliptic curves with a point of order $7$, so the 
desired elliptic curves are the fibers of the universal elliptic 
curve $\E_1$ over the modular curve $X_1(7)$ of genus zero.  
For general $\psi$, we twist 
the elliptic surface $\E_1$ to obtain the appropriate elliptic surface $\E_\psi$, and then 
the $A_v$ are the fibers of $\E_\psi$.
Theorem \ref{genthm} deals with the case where $\psi$ has order 
dividing $3$.  Since any character $\psi$ into $\F_7^\times$ can be written uniquely 
as the product of a character of order dividing $3$ and a character  of order dividing $2$
(namely, $\psi=\psi^4\psi^3$), 
we will obtain the family for a general $\psi$ as a quadratic twist of 
a family with a  cubic $\psi$, in Theorem \ref{moregenthm}.

\begin{defn}
If $E, E'$ are elliptic curves over $k$, and $P \in E(\ksep), P' \in E'(\ksep)$ are points 
of order $7$, we say that $\lambda : (E,P) \isom (E',P')$ is an isomorphism if $\lambda$ 
is an isomorphism from $E$ to $E'$ and $\lambda(P) = P'$.  If such a $\lambda$ 
exists, we say that $(E,P)$ and $(E',P')$ are isomorphic.  
If further $\lambda : E \isom E'$ is defined over $k$, then we say that 
$(E,P)$ and $(E',P')$ are isomorphic over $k$.
\end{defn}

\begin{lem}
\label{extralem}
Suppose $E, E'$ are elliptic curves over $k$, $P \in E(\ksep), P' \in E'(\ksep)$ are points 
of order $7$, and $(E,P)$ is isomorphic to $(E',P')$.
\begin{enumerate}
\item
The isomorphism $\lambda : (E,P) \isom (E',P')$ is unique.
\item
Suppose that the groups $\Psi$ and $\Psi'$ generated by $P$ and $P'$, respectively, 
are stable under $G_k$.  
Then $(E,P)$ and $(E',P')$ are isomorphic over $k$ if and only if the two characters
$$
G_k \to \Aut(\Psi) \isom \F_7^\times, \quad G_k \to \Aut(\Psi') \isom \F_7^\times
$$
are equal.
\end{enumerate}
\end{lem}

\begin{proof}
If $\lambda, \lambda' : (E,P) \isom (E',P')$ are isomorphisms over $\ksep$, 
then $\epsilon = \lambda^{-1}\circ\lambda'$ is an automorphism of $E$ 
fixing $P$, i.e., $(\epsilon-1)(P) = 0$.  But then (viewing $\epsilon$ as 
a root of unity in an imaginary quadratic field) if $\epsilon \ne 1$ we have
$$
7 \le |\ker(\epsilon-1)| = \deg(\epsilon-1) = (\epsilon-1)(\bar\epsilon-1) = 2 - (\epsilon+\bar\epsilon) \le 4
$$
which is impossible.  This proves (i).

For (ii), let $\psi$ and $\psi'$ be the characters giving the action of $G_k$ on $\Psi$ and $\Psi'$, 
respectively.  If $\sigma \in G_k$, then $\lambda^\sigma : E \isom E'$ is an isomorphism, 
and 
$$
\lambda^\sigma(P) = \lambda^\sigma(\psi^{-1}(\sigma)P^\sigma) 
   = \psi^{-1}(\sigma)\lambda(P)^\sigma= \psi^{-1}(\sigma)(P')^\sigma = \psi^{-1}(\sigma)\psi'(\sigma)P'
$$
If $\psi(\sigma) = \psi'(\sigma)$, then $\lambda^\sigma : (E,P) \isom (E',P')$ is an isomorphism, so 
$\lambda^\sigma = \lambda$ by part (i).  On the other hand, if $\psi(\sigma) \ne \psi'(\sigma)$, 
then $\lambda^\sigma(P) \ne \lambda(P)$, so $\lambda^\sigma \ne \lambda$.  This proves (ii).
\end{proof}

If $u \in \ksep$, define a curve $E_u$ over $k(u)$ by
\begin{equation}
\label{univ7}
E_u : y^2 - (u^2-u-1)xy-(u^3-u^2)y = x^3-(u^3-u^2)x^2.
\end{equation}
The discriminant of $E_u$ is
\begin{equation}
\label{Dform}
\Delta(E_u) = u^7(u-1)^7(u^3 - 8u^2 + 5u + 1). 
\end{equation}
The next result is \#15 in Table 3 on p.~217 of \cite{kubert}. 

\begin{thm}[\cite{kubert}]
\label{x1of7thm}
Let $E_u$ be as above.
\begin{enumerate}
\item
If $u \in k$ and $\Delta(E_u) \ne 0$, then $E_u$ is an elliptic curve over 
$k$ and $(0,0)$ is a point of order $7$ in $E(k)$.
\item
If $E$ is an elliptic curve over $k$ and $P\in E(k)$ 
is a point of order $7$ then there is a unique $u \in k$ such 
that $(E,P)$ is isomorphic over $k$ to $(E_u,(0,0))$. 
\end{enumerate}
\end{thm}

Define a linear fractional transformation
\begin{equation}
\label{etadef}
\eta(v)= 1/(1-v).
\end{equation}
The following lemma will be used in the proofs of Theorems \ref{genthm} and \ref{finalfinal} below.  

\begin{lem}
\label{autlem}
Suppose $u \in \ksep$ and $\Delta(E_u) \ne 0$.
Then there is a unique isomorphism defined over $k(u)$:
$$
(E_{\eta(u)},2\cdot(0,0)) \isom (E_u,(0,0)).
$$
\end{lem}

\begin{proof}
A direct computation shows that the map 
$$
(x,y) \;\mapsto\; ((u-1)^4 x+ u^2-u, (u-1)^6 y + (u-1)^4(u^2-2u)x + u^4-2u^3+u^2)
$$
is such an isomorphism.  Uniqueness follows from Lemma \ref{extralem}(i).  
\end{proof}

The following lemma is taken from a paper of Washington \cite[pp.\ 64--65]{washington}.

\begin{lem}[Washington \cite{washington}]
\label{lcwlem}
  Suppose that $K/k$ is a cyclic cubic extension, and $\sigma$ is a generator of 
$\Gal(K/k)$.  Then there is a $t \in k$ such that 
\begin{enumerate}
\item $K$ is the splitting field of the polynomial 
$
f(x) := x^3 - (t+3)x^2 + tx + 1,
$
\item
if $\gamma$ is a root of $f$ then $\gamma^\sigma = \eta(\gamma)$,
where $\eta$ is the linear fractional transformation defined by  \eqref{etadef}.
\end{enumerate}
\end{lem}

\begin{proof}
Choose $\alpha \in K$ such that $K=k(\alpha)$.  The set 
$\{1,\alpha,\alpha^\sigma,\alpha\alpha^\sigma\}$ is linearly dependent over $k$ 
(but $\{1,\alpha\}$ is not), so we can find a linear fractional 
transformation $\phi \in\PGL2(k)$ such that $\alpha^\sigma= \phi(\alpha)$.  
Note that $\phi^3$ fixes $\alpha$, $\alpha^\sigma$, and $\alpha^{\sigma^2}$, 
so $\phi^3 = 1$ in $\PGL2(k)$.

Let $\tilde\phi$ be an element in $\GL2(k)$ whose image under the map $\GL2(k) \to \PGL2(k)$ is $\phi$.   
We must have $\trace(\tilde\phi) \ne 0$.  Otherwise,  $\tilde\phi^2$ would be a scalar matrix and we would then have 
$\alpha^{\sigma^2} = \alpha$.  Therefore we can choose the lift $\tilde\phi$ so that  $\trace(\tilde\phi) = -1$.  Then, writing $I$ for the identity in $\GL2(k)$, 
$$
\tilde\phi^2 + \tilde\phi = -\det(\tilde\phi)I, \quad \tilde\phi^3 = aI
$$
for some $a\in k$, so
$$
(1-\det(\tilde\phi))\tilde\phi = \tilde\phi + \tilde\phi^3 + \tilde\phi^2 = (a-\det(\tilde\phi))I.
$$
Since  $\alpha^\sigma \neq \alpha$,  $\tilde\phi$ cannot be a scalar matrix.  It follows that $\det(\tilde\phi) = 1$ and the minimal polynomial of $\tilde\phi$ over
$k$ is $x^2+x+1$. Let $\tilde\eta  =  \text{\scriptsize $\begin{pmatrix}0&-1\\1&-1\end{pmatrix}$}$, which is a   
lift to $\GL2(k)$ corresponding to 
$\eta$.  
Since $\tilde\eta$ has the same minimal polynomial as $\tilde\phi$, 
there is a $\tilde\xi\in\GL2(k)$ such that 
\begin{equation}
\label{bx}
\tilde\xi\tilde\phi\tilde\xi^{-1} = \tilde\eta.
\end{equation}
Set $\gamma = \xi(\alpha)$, where $\xi$ is the linear fractional transformation corresponding to $\tilde\xi$.    Then $k(\gamma) = k(\alpha) = K$, and by \eqref{bx} we have 
\begin{equation}
\label{gammas}
\gamma^\sigma = \eta(\gamma).  
\end{equation}
Applying $\sigma$ and $\sigma^2$ to  \eqref{gammas}
shows that \eqref{gammas} holds with $\gamma$ replaced by either 
of its conjugates, and that $\gamma^{\sigma^2} = (\gamma-1)/\gamma$.
We compute
$$
\textstyle
(x-\gamma)(x-\gamma^\sigma)(x-\gamma^{\sigma^2})
   = x^3 - (\frac{\gamma^3-3\gamma+1}{\gamma^2-\gamma})x^2 
   + (\frac{\gamma^3-3\gamma^2+1}{\gamma^2-\gamma})x + 1
$$
so the lemma holds with 
$
t := \frac{\gamma^3-3\gamma^2+1}{\gamma^2-\gamma} \in k.
$
\end{proof}

\begin{thm}
\label{genthm}
Suppose that $\chi : G_k \to \F_7^\times$ is a homomorphism, $\chi^3=1$, and  
$E$ is an elliptic curve over $k$. Then $E$ has a $k$-rational subgroup 
of order $7$ on which $G_k$ acts via $\chi$ if and only if there is a $v \in k$ such 
that $E$ is isomorphic over $k$ to the elliptic curve 
$$
A_v : y^2 + a_1(v)xy + a_3(v)y = x^3 + a_2(v)x^2 + a_4(v)x + a_6(v)
$$
over $k$ defined as follows.  
If $\chi = 1$, let
$$
a_1(v) = -v^2+v+1, \quad a_2(v) = a_3(v) = -v^3+v^2, \quad a_4(v) = a_6(v) = 0.
$$
If $\chi \ne 1$, then let $K$ be the cubic extension of $k$ cut out by 
$\chi$, let $\sigma \in \Gal(K/k)$ be the element with $\chi(\sigma) = 4$, 
fix $t \in k$ satisfying Lemma \ref{lcwlem} for $K$ and $\sigma$, 
and let 
\begin{align*}
c &= t^2+3t+9, \quad f(v) = v^3 -(t+3)v^2 + tv + 1, \\
a_1(v) &= c(v^2-v+1), \\
a_2(v) &= cf(v)t(2v - 1), \\
a_3(v) &= cf(v)[(t^3 - 1)v^3 + (t^3 - 1)v + t^2 - t + 1], \\
a_4(v) &= c^2f(v)[(-3t^2 - 5t - 2)v^5 + (2t^3 + 8t^2 + 8t - 7)v^4 \\
   &\hskip15pt - (3t^3 + 6t^2 + 5t - 20)v^3 + (2t^3 - t - 23)v^2 + 2(t^2 + 2t + 7)v - t - 1], \\
a_6(v) &= c^2f(v)^2[(2t^5 + 9t^4 + 23t^3 + 35t^2 + 24t + 11)v^6 \\
   &\hskip15pt + (-t^6 - 6t^5 - 23t^4 - 38t^3 - 33t^2 + 36t)v^5 \\
   &\hskip15pt + (t^6 + 6t^5 + 18t^4 - 6t^3 - 60t^2 - 180t + 13)v^4 \\
   &\hskip15pt + (-t^6 - 2t^5 + 46t^3 + 84t^2 + 142t - 139)v^3 \\
   &\hskip15pt + (-t^5 - 5t^4 - 27t^3 - 15t^2 + 9t + 182)v^2 \\
   &\hskip15pt + (2t^4 + 3t^3 - 10t^2 - 32t - 67)v + 2t^3 + 5t^2 + 11t + 11].
\end{align*}
\end{thm}

\begin{proof}
If $\chi = 1$, then $A_v$ is the curve $E_v$ of \eqref{univ7}, and the 
theorem follows from Theorem \ref{x1of7thm}.

Suppose now that $\chi \ne 1$.
Let $\gamma \in K$ be a root of $f(x)$.  
Define
\begin{align*}
U_v &:= ((\gamma-1)v+1)^2 (2\gamma^2 - (2t + 5)\gamma + t - 1)^2, \\
R_v &:= c f(v) [(\gamma - t - 3) \gamma v - \gamma^2 + (t + 2)\gamma + 1], \\
S_v &:= ((t + 3)\gamma^2 - (t^2 + 5t + 9)\gamma - 3)v^2 \\
   &\hskip30pt - (2t\gamma^2 - 2(t^2 + 3t + 3)\gamma + (t^2 + t + 3))v - 3\gamma^2 + (2t + 6)\gamma - t, \\
T_v &:= c f(v) [(2t^2 + 6t + 5)v^3 - (t^2+3t+9)v^2 - 13v + 2t + 4].
\end{align*}
If $v \in k$ and $A_v$ is 
nonsingular, then we compute that 
$
P_v := (R_v,T_v)
$
is a point of order $7$ in $A_v(K)$, and using Lemma \eqref{lcwlem}(ii) 
we compute that $P_v^\sigma = 4P_v = \chi(\sigma)P_v$.  Thus $P_v$ generates 
a $k$-rational subgroup of order $7$ on $A_v$, on which $G_k$ 
acts via $\chi$.  If $E$ is isomorphic over $k$ to $A_v$, then 
$E$ also has such a subgroup.

Conversely, suppose $E$ is an elliptic curve over $k$ with a $k$-rational 
subgroup of order $7$ on which $G_k$ acts via $\chi$.  Let $P \in E(K)$ 
be a generator of that subgroup 
(so $P^\sigma = \chi(\sigma)P=4P$).  
By Theorem \ref{x1of7thm}(ii) applied with $K$ in place of $k$, 
$(E,P)$ corresponds to a $K$-rational point 
of $X_1(7)$, i.e., there are a $u \in K$ and an isomorphism $\varphi : E \isom E_u$ 
defined over $K$ such that $\varphi(P) = (0,0) \in E_u[7]$.

Let $\delta$ be the linear fractional transformation 
\begin{equation}
\label{deltadef}
\delta(z) = \frac{-z+\gamma}{(\gamma-1)z+1}
\end{equation}
and let $v = \delta^{-1}(u) \in K$.  
We compute that the map $\lambda$ defined by
$$
\lambda(x,y) := (U_v^2 x + R_v, U_v^3 y + U_v^2 S x + T_v)
$$
is an isomorphism over $K$ from $(E_u,(0,0))$ to $(A_v,P_v)$.  
(Since $\delta(v) = u$, by \eqref{deltadef} we have $(\gamma-1)v+1 \ne 0$;
since also $[k(\gamma):k] = 3$, we have $U_v \ne 0$.)  
Therefore $\lambda \circ \varphi$ is an isomorphism from $(E,P)$ to $(A_v,P_v)$.   
If we show that $v \in k$, then Lemma \ref{extralem}(ii) will imply that 
$(E,P)$ and $(A_v,P_v)$ are isomorphic over $k$.

Suppose $\sigma \in G_k$.  Then $\varphi^\sigma$ is an isomorphism from 
$E$ to $E_{u^\sigma}$, and
$$
\varphi^\sigma(P) = \varphi^\sigma(2P^\sigma) = 2\varphi^\sigma(P^\sigma) 
   = 2\varphi(P)^\sigma = 2(0,0)^\sigma = 2(0,0).
$$
Thus we have isomorphisms
$$
(E_u,(0,0)) \map{\;\varphi^{-1}\,} (E,P) \map{\;\varphi^\sigma\,} (E_{u^\sigma},2(0,0)) 
   \map{\;\sim\;} (E_{\eta^{-1}(u^\sigma)},(0,0))
$$
where $\eta$ is defined by \eqref{etadef} and 
the final isomorphism comes from Lemma \ref{autlem}.  Thus by the uniqueness of $u$ in 
Theorem \ref{x1of7thm}(ii) (applied with $K$ in place of $k$) we see that 
$u = \eta^{-1}(u^\sigma)$, so $u^\sigma = \eta(u)$.

Using the definition of $\delta$ and Lemma \ref{lcwlem}(ii), 
it is easy to  check that $\delta^\sigma = \eta\delta$.  
Hence we have 
$$
v^\sigma = \delta^{-1}(u)^\sigma 
   = (\delta^\sigma)^{-1}(u^\sigma) = \delta^{-1} \eta^{-1} (\eta(u)) = \delta^{-1}(u) = v.
$$
Therefore $v \in k$, so $A_v$ is defined over $k$, $G_k$ acts on both $P$ and $P_v$ 
by multiplication by $\chi$, and so the isomorphism $\lambda \circ \varphi : (E,P) \isom (A_v,P_v)$ 
is defined over $k$ by Lemma \ref{extralem}(ii).
\end{proof}

\begin{rem}
\label{gendisc}
Suppose $\chi \ne 1$ in Theorem \ref{genthm}.  With notation as in 
Theorem \ref{genthm}, the discriminant 
of $A_v$ is given by
\begin{equation}
\label{discAveqn}
\Delta(A_v) = c^8 f(v)^7 [(t-5)v^3 + (5t+24)v^2 - (8t+9)v + t-5]. 
\end{equation}
\end{rem}

\begin{rem}
\label{discfrem}
With notation as in Theorem \ref{genthm} with $\chi \ne 1$,
the cubic Galois extension $K$
of $k$ is the splitting field over $k$ of the 
polynomial $f(x) \in k[x]$, by Lemma \ref{lcwlem}(i).  
Thus, $f(x)$ is separable and irreducible over $k$.
One can compute that $c^2$ is the discriminant of 
$f$, so $c \ne 0$.  It then follows from \eqref{discAveqn} that $\Delta(A_v) = 0$ for 
at most six values of $v \in \ksep$.
\end{rem}

\begin{defn}
\label{twistdef}
If $E$ is an elliptic curve over $k$ and $\epsilon : G_k \to \{\pm 1\} \subseteq \Aut(E)$ 
is a homomorphism, then 
the (quadratic) {\em twist} of $E$ by $\epsilon$ is an elliptic curve $E^{(\epsilon)}$ over $k$ such that
there is an isomorphism $\lambda : E^{(\epsilon)} \to E$ over $\ksep$
with $\lambda^\sigma\circ \lambda^{-1}=\epsilon(\sigma)$ for all $\sigma\in G_k$.
\end{defn}

If  $\ch(k) \ne 2$, $E$ is an elliptic curve over $k$ defined by an equation of  the form
$y^2=F(x)$, and $k(\sqrt{d})$ is the 
field cut out by such a character $\epsilon$, then  $E^{(\epsilon)}$  
is isomorphic over $k$ to the curve defined by  $dy^2=F(x)$, i.e., the quadratic twist of $E$ by $d$.

\begin{thm}
\label{moregenthm}
Suppose that $\psi : G_k \to \F_7^\times$ is a homomorphism.  
If $E$ is an elliptic curve over $k$, then $E$ has a $k$-rational subgroup 
of order $7$ on which $G_k$ acts via $\psi$ if and only if there is a $v \in k$ such 
that $E$ is isomorphic over $k$ to the 
twist of $A_v$ by $\psi^3$, where $A_v$ is as in Theorem \ref{genthm} 
for the character $\chi=\psi^4$.

In particular, if $\ch(k) \ne 2$ and $k(\sqrt{d})$ is the 
field cut out by $\psi^3$, then the twist of $A_v$ by $\psi^3$ 
is 
$$
A_v^{(d)} : y^2 = x^3 + d b_2(v) x^2 + 8 d^2 b_4(v) x + 16 d^3 b_6(v),
$$
where $b_2 = a_1^2 + 4a_2$, $b_4 = 2a_4 + a_1 a_3$, $b_6 = a_3^2 + 4a_6$  
are the usual invariants of the curve $A_v$.
\end{thm}

\begin{proof}
Since $\psi^6=1$, we have $(\psi^4)^3=1$, so we can apply Theorem  \ref{genthm},
and we can also  twist by $\psi^3$  as in Definition \ref{twistdef}.
Let $\lambda : E^{(\psi^3)} \isom E$ be as in Definition \ref{twistdef}.  
For $P \in E(\ksep)$ and $\sigma \in G_k$,
$P^\sigma = \psi(\sigma)P$ if and only if   
$\lambda^{-1}(P)^\sigma = \psi^4(\sigma)\lambda^{-1}(P)$.
Thus by  Theorem \ref{genthm},  
\begin{align*}
&\text{$E$ has a $k$-rational subgroup of order $7$ on which $G_k$ acts via $\psi$} \\
   &\quad\Leftrightarrow 
   \text{$E^{(\psi^3)}$ has a $k$-rational subgroup of order $7$  on which $G_k$ acts via $\psi^4$} \\
   &\quad\Leftrightarrow 
   \text{$E^{(\psi^3)}$ is isomorphic over $k$ to $A_v$ for some $v \in k$}\\
   &\quad\Leftrightarrow 
   \text{$E$ is isomorphic over $k$ to $A_v^{(\psi^3)}$ 
      for some $v \in k$}.
\end{align*}
If $\ch(k) \ne 2$, then $A_v^{(d)}$ is a Weierstrass
model for $A_v^{(\psi^3)}$. 
\end{proof}

\begin{lem}
\label{etaautlem}
Let $A_v^{(d)}$ be as in Theorem \ref{moregenthm}, and let $\eta$ be as defined 
by \eqref{etadef}.  Then for every $v$, we have $A_{\eta(v)}^{(d)} \cong A_v^{(d)}$ over $k(v)$.
\end{lem}

\begin{proof}
This can be shown by exhibiting an explicit isomorphism.  We will give a slightly less computational 
way to deduce the lemma from Lemma \ref{autlem}.  We can easily reduce to the case $d = 1$.

Let $\delta$ be the linear fractional transformation defined by \eqref{deltadef} in the 
proof of Theorem \ref{genthm}.  The proof of Theorem \ref{genthm} showed that 
$(A_v,P_v) \cong (E_{\delta(v)},(0,0))$ for every $v$, so we have
$$
(A_v,P_v) \cong (E_{\delta(v)},(0,0)) \cong (E_{\eta\delta(v)},2\cdot(0,0)) 
   \cong (A_{\delta^{-1}\eta\delta(v)},2\cdot P_{\delta^{-1}\eta\delta(v)})
$$
where the middle isomorphism is from Lemma \ref{autlem}.  A simple calculation shows that 
$\delta^{-1}\eta\delta(v) = \eta^2 = \eta^{-1}$, and so $A_v \cong A_{\eta^{-1}(v)}$ over $k(v)$ 
by Lemma \ref{extralem}(ii).
\end{proof}

\begin{cor}
\label{existence}
Suppose that  $\psi : G_k \to  \F_7^\times$ is a homomorphism.
If $k \ne \F_2$,
then there exists an elliptic curve $E$ over $k$ with a
$k$-rational subgroup of order $7$ on which
$G_k$ acts via $\psi$.
If $k = \F_2$, then there exists an elliptic curve
$E$ over $\F_2$ with  an $\F_2$-rational subgroup 
of order $7$ on which $G_{\F_2}$ acts via $\psi$
if and only if  $\psi$ is $\omega^{-1}$ or $\omega^{-1}\epsilon$,
where  $\epsilon$ is the unique character 
of $G_{\F_2}$ of order $2$ 
and $\omega$ is the cyclotomic character.  
\end{cor}

\begin{proof}
Let $A_v$ be as defined in Theorem \ref{genthm}, for the 
(at most cubic) character $\psi^4$.
If $v \in k$ is not a zero of the discriminant $\Delta(A_v)$, then Theorem 
\ref{moregenthm} shows that the twist of $A_v$ by $\psi^3$ 
has the desired property.  
Suppose $k \neq \F_2$.
We need only show that there exists a $v \in k$ 
such that $\Delta(A_v) \ne 0$.

Suppose $\psi^4 = 1$. Then $A_v = E_v$,
so $\Delta(A_v) = \Delta(E_v)$  is given by \eqref{Dform}.
That polynomial has at most $5$ roots, and  
it is easy to check that it has only the 
roots $0$ and $1$ if $|k| \le 5$.  Hence, 
there are $v \in k$ with $\Delta(A_v) \ne 0$ if and only if $k \ne \F_2$.

Now suppose $\psi^4 \ne 1$. Then $\Delta(A_v)$ is the polynomial given 
in \eqref{discAveqn}.  By Remark \ref{discfrem}, $c \ne 0$, 
and $f(v) \ne 0$ for every $v \in k$.  Thus for $v \in k$,  
$\Delta(A_v) = 0$ if and only if  $(t-5)v^3 + (5t+24)v^2 - (8t+9)v + t-5 = 0$.  
If $t \ne 5$, then $\Delta(A_v) \ne 0$ when $v = 0$.  
If $t = 5$, then 
$\Delta(A_v) = 0$ only when $v = 0$ or $1$ (since $\ch(k) \ne 7$), 
so there are $v \in k$ with $\Delta(A_v) \ne 0$ if and only if $k \ne \F_2$.

When $k = \F_2$, the stated result follows from the above proof, the fact that
$\omega$ has order $3$, and the fact that $t=1$ when 
$\psi^4= \omega$ while  $t=0$ when $\psi^4= \omega^{-1}$.  
(Alternatively, it can also be deduced from the fact that no elliptic curve defined over 
$\F_2$ has a rational point of order $7$, which follows from the Weil bounds.)
Note that the remaining characters
are $1$, $\omega$, $\epsilon$, and $\omega\epsilon$;
for each of these characters there is no elliptic curve 
over $\F_2$ with an $\F_2$-rational subgroup of order $7$ on which $G_{\F_2}$ 
acts via that character. 
\end{proof}

\begin{rem}
Here is another interpretation of Theorem \ref{moregenthm}.  
Suppose $\Psi$ is a cyclic group of order $7$ with an action of $G_k$.  
Consider isomorphism classes (in the obvious sense) of pairs $(E,f)$ where $E$ 
is an elliptic curve and $f : \Psi \hookto E[7]$ is an injection.  We say that 
$(E,f)$ is $k$-rational if $(E^\sigma,f^\sigma)$ is isomorphic to $(E,f)$ for every $\sigma \in G_k$,   
and we let $X(\Psi)$ denote the moduli space of such isomorphism classes.  If $K$ is an extension 
of $k$ such that $G_K$ acts trivially on $\Psi$, and $P$ is a generator 
of $\Psi$, then the map 
$$
(E,f) \mapsto (E,f(P))
$$
induces an isomorphism from $X(\Psi)$ to $X_1(7)$ defined over $K$.  Thus $X(\Psi)$ is a 
twist of $X_1(7)$.  

Let $\psi : G_k \to \Aut(\Psi) \cong \F_7^\times$ be the character giving the action 
of $G_k$ on $\Psi$.  
Let $A_v$ be as defined in Theorem \ref{genthm} with $\chi = \psi^4$, 
and let $f_v : \Psi \to A_v[7]$ be the unique homomorphism with $f_v(P) = P_v$, 
where $P_v$ is as in the proof of Theorem \ref{genthm}.
Then Theorem \ref{moregenthm} says that the elliptic surface 
$A_v$ is the universal elliptic curve over $X(\Psi)$, in the following sense.
For every $v \in \ksep$ such that $\Delta(A_v) \ne 0$, the pair $(A_v,f_v)$ is 
$k(v)$-rational, and conversely for every pair ($E,f)$ with $E$ defined over $\ksep$, 
there is a unique $v \in \ksep$ such that $(E,f)$ is isomorphic to $(A_v,f_v)$.

If $a\in \F_7^\times/(\pm1)$ there is a natural automorphism of $X(\Psi)$ induced by 
$(E,f) \mapsto (E,af)$.
The proof of Lemma \ref{etaautlem} shows that the automorphism corresponding to 
$a = 2$ is $\eta$, in the sense that $(A_{\eta(v)},f_{\eta(v)}) \cong (A_v,2f_v)$.
\end{rem}

\section{The special case where  $k = \Q$ and  $\psi = \omega^5$}
\label{speccasesect}

As before, 
let $\omega : G_\Q \to \F_7^\times$ denote 
the cyclotomic character, and  
if $E$ is an elliptic curve defined over $\Q$, 
let $\Deltamin(E)$ denote the discriminant of 
a minimal model of $E$.

Define an elliptic curve
\begin{equation}
\label{defbv}
B_v : y^2 + xy = x^3 - x^2 + \alpha(v)x + \beta(v),
\end{equation}
where
\begin{align*}
\alpha(v) &:= \textstyle
   -\frac{35}{16}v^8 + \frac{63}{4}v^7 - \frac{833}{24}v^6 + \frac{49}{2}v^5 
   + \frac{245}{48}v^4 - \frac{49}{2}v^3 + \frac{343}{24}v^2 + \frac{7}{4}v - 2, \\
\beta(v) &:= \textstyle
   -\frac{49}{32}v^{12} + \frac{637}{48}v^{11} -\frac{ 1617}{32}v^{10} 
   + \frac{44149}{432}v^9 - \frac{16555}{192}v^8
- \frac{1477}{36}v^7 + \frac{1911}{16}v^6 \\
    & \textstyle \hskip 1in
   - \frac{8183}{144}v^5 - \frac{2009}{192}v^4 + \frac{7007}{432}v^3 
   - \frac{147}{16}v^2 + \frac{14}{3}v - 1.
\end{align*}
The discriminant of $B_v$ is 
\begin{equation}
\label{dbv}
\Delta(B_v) = -7^3(v^3 - 2v^2 - v + 1)(v^3 - v^2 - 2v + 1)^7.
\end{equation}

\begin{thm}
\label{omega5}
Let $B_v$ be as above. 
\begin{enumerate}
\item
Suppose $E$ is an elliptic curve over $\Q$.  Then $E$ has a rational subgroup of order $7$ 
on which $G_\Q$ acts via $\omega^5$ if and only if there is a $v \in \Q$ such that 
$E \cong B_v$ over $\Q$.
\item
If $p$ is a prime and $v \in \Q$ is integral at $p$, then 
the above model for $B_v$ is integral at $p$
and minimal at $p$.
\item
If $p$ is a prime and $v \in \Q$ is not integral at $p$, then 
$$
\ord_p(\Deltamin(B_v)) = \ord_p(\Delta(B_v)) - 24\, \ord_p(v).
$$
\end{enumerate}
\end{thm}

\begin{proof}
We will deduce part (i) from Theorem \ref{moregenthm}.
The cyclotomic character $\omega$ has order $6$.
Let $\psi = \omega^5$.  The cubic character $\psi^4 = \omega^2$ 
cuts out the unique cubic subfield $K := \Q(\bmu_7)^+$ of $\Q(\bmu_7)$.
The automorphism 
$\sigma\in\Gal(K/\Q)$ that sends $\zeta_7 + \zeta_7^{-1}$ to $\zeta_7^2 + \zeta_7^{-2}$ 
satisfies $\omega (\sigma)=2$, so $\psi(\sigma) = 2^5=4$.
Using the construction of $\gamma$ and $t$ in the proof of Lemma \ref{lcwlem},
we obtain  $t = -2$ and $\gamma = -(\zeta_7 + \zeta_7^{-1})$.

Further, the quadratic character $\psi^3 = \omega^3$ cuts out the 
unique quadratic subfield $\Q(\sqrt{-7})$ of $\Q(\bmu_7)$.
The map
$(x,y) \mapsto (x', y')$
where
$$
x' = \textstyle \frac{1}{4}x + \frac{3}{4}v^4 - \frac{5}{6}v^3 
   - \frac{15}{4}v^2 + \frac{23}{6}v - 1, \qquad y' = \textstyle \frac{1}{8}y- \frac{1}{2}x'
$$
is an isomorphism from the curve 
$A_v^{(d)}$ of Theorem \ref{moregenthm} with $t = -2$  and $d = -1/7$, 
to the curve $B_v$ of \eqref{defbv} 
(recall that $t$ is in the definition of $A_v$  in Theorem \ref{genthm}, 
which is used to define $A_v^{(d)}$ 
in Theorem \ref{moregenthm}).
Now (i) follows from  Theorem \ref{moregenthm}. 

The polynomials $\alpha(v)$ and $\beta(v)$ take $\Z$ to $\Z$, 
as can be seen, for example, by 
expressing them as integral linear combinations of binomial coefficient polynomials 
$\binom{v}{n}$.
It follows that if $v \in \Q$ is integral at $p$, then the model \eqref{defbv} 
is integral at $p$.  

Suppose first that $p \ne 7$.  
With $c_4$ the usual invariant (see \cite[\S{III.1}]{silverman}), we check that
the polynomial $c_4(B_v)$ is in $\Z[v]$.
The resultant of the polynomials $\Delta(B_v)$ and 
$c_4(B_v)$ is $7^{98}$, 
which is a unit at $p$.  Hence if $v$ is integral at $p$, 
then $\Delta(B_v)$ and $c_4(B_v)$ 
cannot both vanish mod $p$, so by 
Tate's algorithm (or \cite[Proposition III.1.4(ii)]{silverman}), the equation 
defining $B_v$ is minimal at $p$.

If $p=7$ and $v \in \Q$ is integral at $7$, then 
Lemma \ref{tablem}(c) below shows that $\ord_7(\Delta(B_v)) < 12$, 
so the equation defining $B_v$ is minimal at $7$.  This proves (ii).

Now suppose that $v$ is not integral at $p$.  Then $w:=1/v$ is integral at $p$, 
and via the change of variables 
$$
\textstyle
(x,y) \mapsto (w^4x-\frac{w^4-1}{4},w^6y+\frac{w^4(w^2-1)}{2}x+\frac{w^4-1}{8}),
$$
$B_v$ has a model 
\begin{multline*}
\textstyle
\hat{B}_v : y^2 + xy = x^3 - x^2 + (\tilde\alpha(w)+\frac{3}{16}(1-w^{8}))x \\
\textstyle
   + \tilde\beta(w)+\frac{w^{4}-1}{4}\tilde\alpha(w)-\frac{2w^{12}-3w^{8}+1}{64}
\end{multline*}
where $\tilde\alpha(z):=z^8\alpha(1/z)$, $\tilde\beta(z):=z^{12}\beta(1/z) \in \Q[z]$
with $\alpha$ and $\beta$ as in \eqref{defbv}. 
Again, one can check that the polynomials 
$$
\tilde\alpha(z)+3(1-z^{8})/16   \quad  \text{ and }  \quad 
\tilde\beta(z)+(z^{4}-1)\tilde\alpha(z)/4+(-2z^{12}+3z^{8}-1)/16 
$$
take $\Z$ to $\Z$.  Hence  $\hat{B}_v$ is integral at $p$.  
Exactly as for (ii) 
one can show that $\hat{B}_v$ is minimal at $p$, 
and hence  $\Deltamin(B_v) = \Delta(\hat{B}_v) = v^{-24}\Delta(B_v)$.  This proves (iii).
\end{proof}

\begin{rem}
\label{Bvreprmk}
Theorem \ref{omega5}(i) shows that for $v\in\Q$, the representation of $G_\Q$
acting on $B_v[7]$ is of the form $\bigl(\begin{smallmatrix} \omega^5 &\ast\\0&\omega^2 \end{smallmatrix}\bigr)$.
\end{rem}

\begin{cor}
\label{dmincor}
If $v \in \Q$ has denominator $d$, then 
$$
\Deltamin(B_v) = \Delta(B_v)d^{24} = -7^3 d^{24}(v^3 - 2v^2 - v + 1)(v^3 - v^2 - 2v + 1)^7.
$$
\end{cor}

\begin{proof}
This follows directly from Theorem \ref{omega5}(ii,iii).
\end{proof}

\begin{lem}
\label{tablem} 
Let $f_1(v)=v^3-2v^2-v+1$ and $f_2(v)=v^3-v^2-2v+1$.  For $v \in \Q$ and
$B_v$ as above,
we have the following table:
$$
\arraycolsep=2pt
\begin{array}{l|l||c|c|c|c|}
\cline{2-6}
&& \multicolumn{3}{c|}{\text{$v$ integral at $7$}} & \text{$v$ not integral at $7$}\\
\cline{3-5}
&&v \equiv 3 \hskip-7pt \pmod{7} &  v \equiv 5 \hskip-7pt \pmod{7} & \text{otherwise} &   \\
\cline{2-6} 
(\mathrm{a}) & \ord_7(f_1(v)) & 1 & 0 & 0 & 3\,\ord_7(v)  \\
\cline{2-6}
 (\mathrm{b}) & \ord_7(f_2(v)) & 0 & 1 & 0 &  3\,\ord_7(v)  \\
\cline{2-6}
 (\mathrm{c}) & \ord_7(\Delta(B_{v})) & 4 & 10 & 3 & 3+ 24\,\ord_7(v)  \\
\cline{2-6}
 (\mathrm{d}) & \ord_7(\Deltamin(B_{v})) & 4 & 10 & 3 & 3  \\
\cline{2-6}
 (\mathrm{e}) & \ord_7(j(B_v)) & \ge 2 & \ge 5 &  0 & 0 \\
\cline{2-6}
\end{array}
$$
\end{lem}

\begin{proof}
If $v$ is not integral at $7$, then 
$\ord_7(f_1(v))= \ord_7(f_2(v))= 3\ord_7(v)$.
If $v$ is integral at $7$, then a direct computation shows that 
$f_1(v), f_2(v) \not\equiv 0 \pmod{7^2}$.
Since $f_1(v) \equiv (v-3)^3 \pmod{7}$ and $f_2(v) \equiv (v-5)^3 \pmod{7}$, 
{(a)} and {(b)} follow. 

By \eqref{dbv} we have $\Delta(B_{v})=-7^3f_1(v)f_2(v)^7$, so (c)
follows from {(a)} and {(b)}.
Assertion {(d)} follows 
directly from {(c)} and Theorem \ref{omega5}(ii,iii).

We compute that
\begin{equation}
\label{jBveqn}
j(B_v) = - \frac{[(v^2 - 3v - 3)(v^2 - v + 1)(3v^2 - 9v + 5)(5v^2 - v - 1)]^3}{
   f_1(v)f_2(v)^7}.
\end{equation}
If $v \equiv 5 \pmod{7}$, then
each quadratic factor in the numerator vanishes modulo $7$.  
If $v \equiv 3 \pmod{7}$, then 
$v^2 - v + 1 \equiv 0 \pmod{7}$.  If $v \not\equiv 3$ or $5 \pmod{7}$, then none of the 
factors in the numerator vanish modulo $7$.  
These remarks together with {(a)} and {(b)} imply the assertions in {(e)}.
\end{proof}

\begin{prop}
\label{isogprop}
For $v \in \Q$ and $B_v$ as above,  
let $\Psi_v$ be the $\Q$-rational subgroup of $B_v$ of order $7$ on which 
$G_\Q$ acts via $\omega^5$.  Let $B'_v$
be the quotient of $B_v$ by $\Psi_v$, so there is 
an isogeny from $B_v$ to $B'_v$ defined over $\Q$.  
Then the isogenous curve $B'_v$ is isomorphic over $\Q$ 
to the twist of $B_{1-v}$ by $\omega^3$.
\end{prop}

\begin{proof}
One can verify this by a direct calculation, using the formulas of V\'elu \cite{velu} 
(see \cite[\S4.1]{bostan}) to exhibit the isogeny. 
(See especially Proposition 4.1 of \cite{bostan} and the formulas for $\tilde{A}$ and $\tilde{B}$
in the paragraph after its proof.)
\end{proof}

Note that $B'_v$ 
has a subgroup of order $7$ on which $G_\Q$ acts via $\omega^2$,
namely, $B_v[7]/\Psi_v$.  
Also, the twist of $B_{1-v}$ by $\omega^3$ is
the quadratic twist  of $B_{1-v}$ by $-7$.

\begin{cor}
\label{dmbb}
Suppose $v \in \Q$.  Then:
\begin{enumerate}
\itemsep=5pt
\item
$
\displaystyle
\frac{\Deltamin(B'_v)}{\Deltamin(B_v)} = 7^{s_v}\biggl(\frac{v^3-2v^2-v+1}{v^3-v^2-2v+1}\biggr)^6
$
for some $s_v \in \{ \pm 6 \}$, and
\item
$
\displaystyle
\ord_7\biggl(\frac{\Deltamin(B'_v)}{\Deltamin(B_v)}\biggr) = 
\begin{cases}
0 & \text{if $v$ is integral at $7$ and $v \equiv 3$ or $5 \pmod{7}$,} \\
6 & \text{otherwise.}
\end{cases}
$
\end{enumerate}
\end{cor}

\begin{proof}
Let $f_1(v)=v^3-2v^2-v+1$ and $f_2(v)=v^3-v^2-2v+1$ as in Lemma \ref{tablem},
let $B^{(-7)}_{1-v}$ denote the quadratic twist 
of $B_{1-v}$ by $-7$, and let 
\begin{equation}
\label{j}
s_v := \ord_7\biggl(\frac{\Deltamin(B^{(-7)}_{1-v})}{\Deltamin(B_{1-v})}\biggr) \in \{ \pm 6 \}.
\end{equation}

Note that $f_1(1-v)=-f_2(v)$ and $f_2(1-v)=-f_1(v)$.
Since $v$ and $1-v$ have the same denominator, 
by Corollary \ref{dmincor} we have 
$$\Deltamin(B_{1-v})/\Deltamin(B_{v}) = f_1(v)^6/f_2(v)^6.$$
By Proposition \ref{isogprop} we have $B'_v \cong B^{(-7)}_{1-v}$.
Thus
\begin{equation}
\label{j2}
\frac{\Deltamin(B'_v)}{\Deltamin(B_v)} = 
\frac{\Deltamin(B^{(-7)}_{1-v})}{\Deltamin(B_{v})} =
\frac{7^{s_v} \Deltamin(B_{1-v})}{\Deltamin(B_{v})} =
7^{s_v}\biggl(\frac{f_1(v)}{f_2(v)}\biggr)^6,
\end{equation}
proving (i).

To prove (ii) we need to compute $s_v$.  Using Lemma \ref{tablem}(e), 
it follows that
$$
\ord_7\big(j(B^{(-7)}_{1-v})\big) = \ord_7\big(j(B_{1-v})\big) \ge 0.
$$
Hence,  by Tate's algorithm (see for example \cite[Table 15.1]{silverman}), we have 
$$
0 \le \ord_7\big(\Deltamin(B^{(-7)}_{1-v})\big) \le 10.
$$  
It follows that the integer $s_v$ 
of \eqref{j} satisfies
$$
s_v = \begin{cases}\phantom{-}6&\text{if $\ord_7\big(\Deltamin(B_{1-v})\big)<6$,}\\
   -6&\text{if $\ord_7\big(\Deltamin(B_{1-v})\big)\ge6$.}\end{cases}
$$
Now Lemma \ref{tablem}(d) shows that $s_v = -6$ if  $v$ is both integral 
at $7$ and congruent to $3 \pmod{7}$, and $s_v = 6$ otherwise.
Assertion (ii) now follows from \eqref{j2} and Lemma \ref{tablem}(a,b).
\end{proof}

\begin{prop}
\label{redtype}  
Suppose that $v \in \P(\Q)$. Then:
\begin{enumerate}
\item
The conductor of the elliptic curve $B_v$ is of the form
$
49\prod^{m}_{i=1}\ell_i,
$
where  the $\ell_i$'s are distinct primes such that $\ell_i \equiv \pm 1 \pmod{7}$.  
\item
 The curve $B_v$ has potentially good reduction at $7$.  Further, 
if $v$ is integral at $7$ and $v \equiv 3$ or $5 \pmod{7}$,
then $B_v$ has potentially ordinary reduction at $7$, and for all other $v \in \P(\Q)$, 
$B_v$ has potentially supersingular reduction at $7$.   
\item
The conductor of $B_v$ is $49$ if and only if $v \in \{0,1,\infty, 2, 1/2, -1\}$.
\end{enumerate}
\end{prop}

\begin{proof}   
Lemma \ref{tablem}(e) shows that $j(B_v)$ is integral at $7$ for all $v \in \P(\Q)$.  
Hence $B_v$ always has potentially good reduction at $7$, giving (ii). However, $B_v$ 
cannot have good reduction at $7$. One sees this by considering the action of $G_{\Q_7}$ 
on $B_v[7]$.  If $B_v$ had good ordinary reduction at $7$, then $B_v[7]$ would have a 
nontrivial unramified quotient over $\Q_7$ (by \cite[Proposition 11]{serre}), which 
is not the case since $\omega^2$ and $\omega^5$ are ramified characters of $G_{\Q_7}$.  
If $B_v$ had good supersingular reduction, then $B_v[7]$ would be irreducible  over $\Q_7$
(by \cite[Proposition 12(c)]{serre}), which is also not the case. It follows that 
the conductor of $B_v$ is $49M$, where $M$ is not divisible by $7$.   

By examining the elliptic curves over $\F_7$,
we see that an elliptic curve $E$ over $\Q$ has supersingular or potentially 
supersingular reduction at $7$ if and only if $j(E) \equiv -1 \pmod{7}$.  
If $v$ is integral at $7$ and  satisfies   $v \equiv 3$ or $5 \pmod{7}$, then  
Lemma \ref{tablem}(e) shows that $j(B_v) \equiv 0 \pmod{7}$, and so $B_v$ has 
potentially ordinary reduction at $7$.  For the other $v$'s, the formula for 
$j(B_v)$ in \eqref{jBveqn} shows that we indeed have $j(B_v) \equiv -1 \pmod{7}$.
This proves (ii).

Suppose $\ell$ is a prime dividing $M$.  If $B_v$ has additive reduction at $\ell$, 
then $B_v$ becomes semistable over $\Q(B_v[7])$ and the ramification degree of 
$\ell$ in $[\Q(B_v[7]):\Q]$ is divisible by $2$ or $3$. (A good summary of the ramification 
properties in the non-semistable case can be found in \cite[\S5.6]{serre},
especially Proposition 23(b).)  
This contradicts the facts that $\ell$ is unramified in $\Q(\bmu_7)/\Q$ and 
$[\Q(B_v[7]):\Q(\bmu_7)]$ is (by Remark \ref{Bvreprmk}) $1$ or $7$. Thus, $B_v$ 
has multiplicative reduction at $\ell$.  It follows that $M$ is not divisible by $\ell^2$.    

The action of $G_{\Q_{\ell}}$ on $B_v[7^{\infty}]$ can be described by the Tate 
parametrization. One sees that  $B_v[7]$, or an unramified quadratic twist of 
$B_v[7]$,  has composition factors isomorphic over $\Q_\ell$ to $\bmu_7$ and $\Z/7\Z$. Let $\omega_{\ell}$ denote the restriction of $\omega$ to  $G_{\Q_{\ell}}$.  
Thus, $G_{\Q_{\ell}}$ acts on the composition factors by two $\F_7^{\times}$-valued 
characters whose ratio is  $\omega_{\ell}$, or its inverse. On the other hand, 
$G_{\Q_{\ell}}$ acts on the composition factors via  $\omega_{\ell}^2$ and  $\omega_{\ell}^5$, whose 
ratio is $\omega_{\ell}^{\pm 3}$,  a character of order $1$ or $2$.  Therefore 
$\omega_{\ell}$ 
has order $1$ or $2$, so $\ell \equiv \pm 1 \pmod{7}$, giving (i).    

There are exactly two $j$-invariants of curves of conductor $49$.  Using \eqref{jBveqn} 
we see that these correspond precisely to the six values of $v$ listed in (iii).
\end{proof}

\begin{rem}\label{49vs}  There is an $S_3$-action on $\P(\Q)$ defined by the linear 
fractional transformations $\eta$ of \eqref{etadef} and 
$\tau$ defined by $\tau(v) = 1-v$. Since the fixed points of $\eta$ (the 
primitive sixth roots of unity) are not in $\Q$, the orbits under the action of $\eta$ 
always have length 3.  There are just two orbits of length 3 under 
the action of $S_3$.   For $v$ is in such an orbit if and only if $1-v = \eta^i(v)$ 
for some $i \in \{0,1,2\}$. One easily determines the possible orbits of that type: 
$\{0,1,\infty\}$ is one, $\{2, 1/2, -1\}$ is the other.  
By Proposition \ref{redtype}(iii) the corresponding curves $B_v$ have 
conductor $49$, and hence have complex multiplication.  One can also explain this as follows.

By Proposition \ref{isogprop}, the elliptic curves  $B_{1-v}^{(-7)}$ and  $B_v/\Psi_v$ are $\Q$-isomorphic for 
every $v \in \P(\Q)$. However, if $1-v \in \{v, \eta(v), \eta^2(v)\}$, then 
Lemma \ref{etaautlem} shows that $B_{1-v}$ is 
$\Q$-isomorphic to $B_{v}$, so there is an isomorphism 
$B_v/\Psi_v \cong B_{v}$ defined over $F = \Q(\sqrt{-7})$. Therefore, $B_v$ has 
an endomorphism of degree 7 defined over $F$.  This means that $B_v$  has 
CM by $F$.  Furthermore, since a CM curve has no primes of multiplicative reduction, 
Proposition \ref{redtype} shows that the the conductor of $B_v$ is $49$.  
If $v \in \{0,1,\infty\}$, then $j(B_v) = -15^3$ and $\End(B_v)$ is the maximal order in $F$.  
If $v \in \{2, 1/2, -1\}$, then $j(B_v) = 255^3$ and  $\End(B_v)$ is the nonmaximal order  
of conductor $2$ in $F$.
\end{rem}

\section{The image of $\rho_{E,7}$}    
\label{image7sect}
Suppose $E$ is an elliptic curve over $\Q$ with a $\Q$-isogeny 
of prime degree $p \ge 7$.  We retain the notation of \S\ref{rhosect}.  
Note that since $\psi\varphi = \omega$, we have that $\psi\varphi^{-1}$ has order $2$ 
if and only if $\psi^4 = \omega^2$.

By Theorem \ref{pthm} (\cite[Theorem 1]{greenberg}), if $E$ is $p$-exceptional then 
$\psi\varphi^{-1}$ has order $2$.  If $p > 7$ then Proposition \ref{pprop} 
(\cite[Remark 4.2.1]{greenberg}) says that if $\psi\varphi^{-1}$ has order $2$ 
then $E$ has CM by $\Q(\sqrt{-p})$.  
If $E$ has CM, then $\im(\rho_{E,p})$ is a $p$-adic Lie group of dimension 
$2$, and so it cannot contain a Sylow pro-$p$ subgroup of $\Aut(T_p(E))$.  
Thus for $p > 7$, an elliptic curve over $\Q$ is $p$-exceptional 
if and only if it has CM by $\Q(\sqrt{-p})$.
 
However, for $p=7$, it is possible for $\psi\varphi^{-1}$ to have order 2 even if 
$E$ does not have complex multiplication.  For example, for every $v \in \Q$,  
the curve $B_v$ of \S\ref{speccasesect} has a $\Q$-isogeny of degree 7 with 
$\psi = \omega^5$ and $\psi\varphi^{-1} = \omega^3$ of order $2$.
In this section we will use Theorem \ref{Deltacriterion} to study 
$7$-exceptional curves.  We assume from now on that $p=7$.

For $j \in \Z$, let  $C_j$ denote the curve
\begin{equation}
\label{cjdef}
w^7 = 7^{j}\biggl(\frac{v^3-2v^2-v+1}{v^3-v^2-2v+1}\biggr).
\end{equation}

\begin{lem}
\label{cjempty}
Let $C_j$ be as above.
\begin{enumerate}
\item
For every $j\in\Z$, the curve $C_j$ has genus $12$.  
\item
If $7 \nmid j$, then $C_j(\Q) = \emptyset$.
\end{enumerate}
\end{lem}

\begin{proof}
The curves $C_j$ are degree 7 covers of $\P(\C)$ with six branch points, each with ramification 
degree 7, so by the Riemann-Hurwitz formula they have genus $12$. 

To prove (ii), we will show that $C_j(\Q_7) = \emptyset$ if $7 \nmid j$.  
The map $(v,w) \mapsto (1/v, 1/w)$ defines an isomorphism from $C_j$ to $C_{-j}$.
Since $C_j \cong C_{j'}$ if $j \equiv j' \pmod{7}$, 
it suffices to consider just $j = 1$, $2$, and $3$.
Then if $(v,w)$ is a point on $C_j$ or $C_{-j}$ defined over $\Q_7$, and $v \in \Z_7$, 
then $w \in \Z^{\times}_7$; this follows immediately from  Lemma \ref{tablem}(a,b), 
which is clearly valid even for $v \in \Q_7$.  
Thus, to show that $C_j(\Q_7)$ is empty, it suffices to show that neither of 
the equations
$$
\renewcommand{\arraystretch}{1.5}
\begin{array}{c}
v^3-2v^2-v+1 = 7^jw^7(v^3-v^2-2v+1), \\   
v^3 - v^2-2v+1 = 7^jw^7(v^3-2v^2-v+1)\phantom{,}
\end{array}
$$
has solutions $v, w \in \Z_7$.  
If $j = 2$ or $3$, this follows from Lemma \ref{tablem}(a,b) since the powers of $7$ on the 
two sides differ.  If $j = 1$, then one finds easily that neither  equation
has a solution modulo $7^3$.  
\end{proof}

Recall the elliptic curve $B_v$ defined by \eqref{defbv},
and the definition of $p$-exceptional in Definition \ref{excdef}.

\begin{prop}
\label{gclem} 
Suppose $v \in \Q$.  Then the following are equivalent:
\begin{enumerate}
\item 
$B_v$ is $7$-exceptional;
\item
there is a $w \in \Q$ such that $(v,w) \in C_0(\Q)$.   
\end{enumerate} 
\end{prop}

\begin{proof}
By Proposition \ref{redtype}, $B_v$ has multiplicative reduction at all primes 
of bad reduction different from $7$.
Thus we can apply Theorem \ref{Deltacriterion} to $B_v$ with $p=7$ 
to conclude that 
\begin{align*}
\text{Assertion (i)} &\Longleftrightarrow 
   \frac{\Deltamin(B'_v)}{\Deltamin(B_v)} \in 7^\Z \cdot (\Q^{\times})^7 \\
&\Longleftrightarrow \biggl(\frac{v^3-2v^2-v+1}{v^3-v^2-2v+1}\biggr)^6 \in 7^\Z \cdot (\Q^{\times})^7 \\
&\Longleftrightarrow \frac{v^3-2v^2-v+1}{v^3-v^2-2v+1} \in 7^\Z \cdot (\Q^{\times})^7,
\end{align*} 
the middle equivalence by  Corollary \ref{dmbb}(i).
This in turn is equivalent to saying there is a point $(v,w) \in C_j(\Q)$ for some $j$ with $0 \le j \le 6$.  But 
by Lemma \ref{cjempty}(ii), $C_j(\Q)$ is empty if $1 \le j \le 6$.  This proves the lemma.
\end{proof}

\begin{thm} 
\label{goodcriterion}  Suppose that $E$ is an elliptic curve over $\Q$ with a $\Q$-isogeny of degree 7.   
Then the following are equivalent:
\begin{enumerate}
\item  
$E$ is $7$-exceptional;
\item
$E$ is a quadratic twist of $B_v$ for some $(v,w) \in C_0(\Q)$.   
\end{enumerate} 
\end{thm}

\begin{proof} 
Suppose (i),
i.e., the image of $\rho_{E,7}$ does not contain a Sylow pro-$7$ subgroup of $\Aut_{\Z_7}(T_7(E))$.  
By Theorem \ref{pthm} (\cite[Theorem 1]{greenberg}), $\psi\varphi^{-1}$ has order 2, so by Lemma \ref{ratio} we have 
$\psi\varphi^{-1} = \omega^3$.
Since $\psi\varphi = \omega$, we have $\psi^2 = \varphi^2 = \omega^4$.  Let $\epsilon = \psi\omega$. 
Then $\epsilon$ is a quadratic 
character, $\psi = \omega^5\epsilon$, and $\varphi = \omega^2\epsilon$.  Replacing $E$ by its 
quadratic twist by $\epsilon$, we may assume that $\psi = \omega^5$ and $\varphi = \omega^2$.
By Theorem \ref{omega5}(i), we have that $E \cong B_v$ for some $v \in \Q$.  
Now the theorem follows from Proposition \ref{gclem}.
\end{proof}

The curve 
$$
C_0 \quad  :  \quad  w^7 = \frac{v^3-2v^2-v+1}{v^3-v^2-2v+1}
$$
 of \eqref{cjdef} has a nonsingular model $C \subset \P \times \P$ 
with coordinates $((v:u),(w:z))$ (that we will abbreviate as $(v,w)$) given by
$$
w^7(v^3-v^2u-2vu^2+u^3) = z^7(v^3-2v^2u-vu^2+u^3),
$$
which has good reduction outside of $7$.  
By Theorem \ref{goodcriterion}, we wish to determine $C(\Q)$.  
One finds easily the following rational points on~$C$.
Let
\begin{gather*}
   P_0 = (0, 1), \quad P_1 = (1, 1), \quad P_\infty = (\infty, 1), \\
   P_2 = (2, -1), \quad P_3 = (\tfrac{1}{2}, -1), \quad P_4 = (-1, -1)
\end{gather*}
and
$$
Z = \{P_0,P_1,P_\infty,P_2,P_3,P_4\} \subseteq C(\Q).
$$

\begin{thm}
\label{final}
$C(\Q) = Z$.
\end{thm}

We will prove Theorem \ref{final} using the method of Chabauty, as made explicit 
in \cite{StollChabauty}.  Before that we deduce the following consequence.

\begin{thm}
\label{finalfinal}
If $E$ is an elliptic curve over $\Q$ with a $\Q$-rational subgroup of order $7$,
and $E$ is exceptional for $7$,
then $E$ has CM by $\Q(\sqrt{-7})$, i.e., $j(E) \in \{-15^3,255^3\}$,
i.e., $E$ is a quadratic twist of an elliptic curve 
of conductor $49$.
\end{thm}

\begin{proof}
By Theorem \ref{goodcriterion}, $E$ is 
a quadratic twist of $B_v$ for some $(v,w) \in C(\Q)$.  
By Theorem \ref{final},  $(v,w) \in Z$. 
If $v \in \{0,1,\infty\}$,  then 
$B_v$ is isomorphic to the curve 49A1 in Cremona's tables \cite{cremona}.  
If $v \in \{2,1/2,-1\}$,  then $B_v$ is isomorphic to the curve 49A2.  
\end{proof}

\section{The rank of $J(\Q)$}
\label{ranksect}
Let $J$ be the Jacobian of $C$. The first step in bounding $C(\Q)$ is to compute the 
rank of $J(\Q)$.  In this section we prove that the rank is $6$.
We first prove that $6$ is an upper bound, 
following the method described by Poonen and Schaefer in \cite{PooSch}.
To keep our notation as close as possible to theirs, we replace $C$ by the 
(birationally) isomorphic curve
$$
X : y^7 = (x^3-2x^2-x+1)(x^3-x^2-2x+1)^6.
$$

Let $\zeta$ be a primitive $7$-th root of unity, $k = \Q(\zeta)$, 
$\O = \Z[\zeta]$, 
and $\pi = \zeta-1 \in \O$, a generator of the prime ideal of $\O$ above $7$.  
We identify $\O$ with a subring of $\End_k(J)$ by sending $\zeta$ to the automorphism of $J$ 
induced by the automorphism $(x,y) \mapsto (x,\zeta y)$ of $X$.  We will use \cite{PooSch} 
to compute an upper bound for the size of $J(k)/\pi J(k)$. 

Define 
\begin{align*}
f(x) &= (x^3-2x^2-x+1)(x^3-x^2-2x+1)^6, \\
f_0(x) &= (x^3-2x^2-x+1)(x^3-x^2-2x+1).
\end{align*}
A calculation in PARI/GP shows that the roots of $x^3-2x^2-x+1$ are 
$\alpha_i := 1+\zeta^i+\zeta^{-i} \in k$ for $1 \le i \le 3$, and the roots of 
$x^3-x^2-2x+1$ are $\alpha_i := -\zeta^i-\zeta^{-i} \in k$ for $4 \le i \le 6$.  
In particular, $f$ and $f_0$ factor into linear factors in $k[x]$.  

Suppose $K$ is a field containing $k$.  Let $\Div(X/K)$ denote the group of 
$K$-rational divisors on $X$, i.e., the group of $\Z$-linear combinations of 
points in $X(\bar{K})$ that are fixed by $G_K$, let $\Div ^0(X/K)$ denote the subgroup 
of divisors of degree zero, and let $\Pic^0(X/K) = \Div^0(X/K)/P(X/K)$
where $P(X/K)$ is the group of divisors of $K$-rational functions on $X$
(i.e., the principal divisors).  Since $X(k)$ is nonempty, 
there is a natural isomorphism $\Pic^0(X/K) \cong J(K)$, and we will identify these 
two groups.

If $R$ is a (multiplicative) abelian group, let
$$
V(R) := (R/R^7)^6/(R/R^7)
$$
where $R^7$ denotes seventh powers in $R$,
and $R/R^7$ is embedded diagonally in the direct product $(R/R^7)^6$.

In \cite[\S5]{PooSch}, Poonen and Schaefer define what they call the ``$(x-T)$ map'' for every 
field $K$ containing $k$:
$$
(x-T)_K : J(K)/\pi J(K) \too V(K^\times).
$$
This map is characterized as follows.
If $D = \sum_P n_P P \in \Div^0(X)$ is supported on points 
$P \in X(\bar{K})$ with $x$-coordinate 
$x(P) \notin \{\alpha_i : 1 \le i \le 6\} \cup \{\infty\}$, then 
$$
(x-T)_K(D) := \prod_P((x(P)-\alpha_1)^{n_P},\ldots,(x(P)-\alpha_6)^{n_P}).
$$

\begin{lem}[Poonen-Schaefer \cite{PooSch}]
\label{xmTpt}
Suppose $K$ is a field containing $k$, and $P = (x(P),y(P)) \in X(K)$.  Let $\infty$ 
denote the rational point with $x = \infty$, i.e., the point corresponding to $(\infty, 1)$ 
on the nonsingular model $C$ of $X$.  
If $x(P) \notin \{\alpha_i : 1 \le i \le 6\} \cup \{\infty\}$ then 
$$
(x-T)_K(P-\infty) = (x(P)-\alpha_1, \ldots, x(P)-\alpha_6).
$$
\end{lem}

\begin{proof}
This follows from  \cite[Proposition 5.1]{PooSch}.
\end{proof}

There is a natural 
localization map from $V(k^\times)$ to $V(k_\pi^\times)$,
where $k_\pi$ is the completion of $k$ at $\pi$.
Let $N$ be the 
``weighted norm" map  from \S 6 of \cite{PooSch}:
$$
N : V(k^\times) \to k^\times/(k^\times)^7, \quad (z_1,z_2,z_3,z_4,z_5,z_6) \mapsto  z_1z_2z_3(z_4z_5z_6)^6.
$$

\begin{thm}[Poonen-Schaefer \cite{PooSch}]
\label{PS}
In the commutative diagram
$$
\xymatrix@C=40pt{
~J(k)/\pi J(k)~ \ar^-{(x-T)_k}[r]\ar[d] & V(k^\times) \ar^{\loc_\pi}[d]~ \\
~J(k_\pi)/\pi J(k_\pi)~ \ar^-{(x-T)_{k_\pi}}[r] & ~V(k_\pi^\times)~
}
$$
the maps $(x-T)_k$ and $(x-T)_{k_\pi}$ are injective, and the image of $(x-T)_k$ is contained in
$$
V(\O[1/\pi]^\times) \;\cap\; \ker(N) \;\cap\; \loc_{\pi}^{-1}(\im((x-T)_{k_\pi})).
$$
\end{thm}

\begin{proof}
That the maps are injective follows from \cite[Theorem 11.3]{PooSch}, since $X$ has 
$k$-rational points and $f(x)$ factors into linear factors in $k[x]$.

Let $U$ denote the image of $(x-T)_k$.  Then $U \subseteq V(\O[1/\pi]^\times)$ by \cite[Proposition 12.4]{PooSch}  
since $J$ has good reduction outside of $7$, and $U \subseteq\ker(N)$ by \cite[Proposition 12.1]{PooSch}.
The commutativity of the diagram shows that $U \subseteq \loc_{\pi}^{-1}(\im((x-T)_{k_\pi}))$. 
\end{proof}

\begin{lem}[Poonen-Schaefer \cite{PooSch}]
\label{pslem}
We have
\begin{enumerate}
\item
$\dim_{\F_7}J(k)[\pi] = 4$,
\item
$\dim_{\F_7}J(k_\pi)/\pi J(k_\pi) = 16$.
\end{enumerate}
\end{lem}

\begin{proof}
Assertion (i) is \cite[Lemma 12.9]{PooSch}, since $f(x)$ factors into linear factors in $k[x]$, 
of which $6$ are distinct.

Similarly, \cite[Lemma 12.9]{PooSch} shows that $\dim_{\F_7}J(k_\pi)[\pi] = 4$, 
and then \cite[Lemma 12.10]{PooSch} shows that 
$$
\dim_{\F_7}J(k_\pi)/\pi J(k_\pi) = g + \dim_{\F_7}J(k_\pi)[\pi] = 12 + 4 = 16,
$$
where $g = 12$ is the genus of $X$. 
\end{proof}

\begin{rem}
\label{SigmaDefrem}
We  observe that $C$ has an action of the group $\Sigma \cong S_3$
generated by the two involutions
\[   (v, w) \longmapsto (v^{-1}, w^{-1}) \quad\text{and}\quad
     (v, w) \longmapsto (1-v, w^{-1}).
\]
The group $\Sigma$ has the two orbits $\{P_0, P_1, P_\infty\}$ 
and $\{P_2, P_3, P_4\}$ on the
set $Z$ of known rational points.
\end{rem}

\begin{prop}
\label{rkO}
$\rk_\O J(k) \le 6$.
\end{prop}

\begin{proof}
We will use Theorem \ref{PS} to bound the $\O$-rank of $J(k)$.  All of the terms in 
Theorem \ref{PS} are $\F_7$-vector spaces, and we need to compute them explicitly.

It follows from Theorem 5.1 of Chapter 3 of \cite{langcycl},
and the fact that $\Q(\zeta + \zeta^{-1})$ has class number one,
that $\O[1/\pi]^\times$ is generated by the roots of unity, 
the cyclotomic units, and $\pi$.  
Thus an $\F_7$-basis of $\O[1/\pi]^\times/(\O[1/\pi]^\times)^7$ is given by 
$\{\zeta, 1+\zeta, 1+\zeta+\zeta^2, \pi\}$.

We need to compute the image of $(x-T)_{k_\pi}$.  
By Theorem \ref{PS} and Lemma \ref{pslem}(ii), 
$\dim_{\F_7}(\im((x-T)_{k_\pi})) = 16$.  Using PARI/GP, we find  
points $Q_i=(x_i,y_i) \in X(k_\pi)$ for $1\le i\le 6$ with $x$-coordinates:
\begin{align*}
x_1 &= 0, \quad x_2 = -1, \\
x_3 &= 3+4\pi^2+5\pi^3+\pi^4+4\pi^5+2\pi^6+6\pi^7+5\pi^8+5\pi^9+5\pi^{10}, \\
x_4 &= 3+\pi^2+5\pi^3+5\pi^4+5\pi^5+5\pi^6+2\pi^8+5\pi^9+\pi^{10}, \\
x_5 &= 3+\pi^2+2\pi^4+4\pi^5+2\pi^6+\pi^7+2\pi^8, \\
x_6 &= 3+2\pi^2+5\pi^3+\pi^4+6\pi^7+2\pi^8+6\pi^9+2\pi^{10} \\
    & \hskip1.5in + 5\pi^{11}+4\pi^{12}+2\pi^{14}+2\pi^{15}+6\pi^{16}+\pi^{17}.
\end{align*}
Using PARI/GP and Lemma \ref{xmTpt}, we compute $(x-T)_{k_\pi}(\sigma(Q_i)-\infty)$ 
for $1 \le i \le 6$ and for all $\sigma \in \Sigma$, and we find that those values generate an 
$\F_7$-subspace of $V(k_\pi^\times)$ of dimension $16$.  
(We work inside the $\F_7$-vector space $k_\pi^\times/(k_\pi^\times)^7$, using the basis 
$$
\{\pi,1+\pi,1+\pi^2,1+\pi^3,1+\pi^4,1+\pi^5,1+\pi^6,1+\pi^7\}.)
$$ 
It follows that we have found the full image of $(x-T)_{k_\pi}$.

Using the above information, a linear algebra computation in PARI/GP now shows that 
$$
\dim_{\F_7} (V(\O[1/\pi]^\times) \cap \ker(N) \cap \loc_{\pi}^{-1}(\im((x-T)_{k_\pi}))) = 10.
$$
Therefore by Theorem \ref{PS} we have $\dim_{\F_7}J(k)/\pi J(k) \le 10$.  
Since 
$$
\dim_{\F_7}J(k)/\pi J(k) = \rk_\O J(k) + \dim_{\F_7}J(k)[\pi], 
$$
and $\dim_{\F_7}J(k)[\pi] = 4$ by Lemma \ref{pslem}(i), we conclude that $\rk_\O J(k) \le 6$.
\end{proof}

\begin{lem}
\label{rkQlem}
$\rk_\Z J(\Q) \le 6$.
\end{lem}

\begin{proof}
This follows from Proposition \ref{rkO} and \cite[Lemma 13.4]{PooSch}.
\end{proof}

\begin{rem}
\label{Ddefrem}
The involution $(v, w) \mapsto (v^{-1}, w^{-1})$ 
has exactly the two fixed points $P_1$ and~$P_4$,
therefore the quotient of $C$ by this involution is a curve~$D$ of genus~6.
The Jacobian~$J$ of~$C$ is isogenous to a product of two copies of the Jacobian
of~$D$. 
The genus $6$ curve $D$ is 
$$
Y^7 - 7Y^5 + 14Y^3 - 7Y = (2X^3 - 6X^2 - 7X + 24)/(X^3 - 3X^2 - 4X + 13)
$$
which by a change of variables is
$$
(-2 + Y)(-1 - 2Y + Y^2 + Y^3)^2 = X/(1 - 4X + 3X^2 + X^3).
$$
\end{rem}

\begin{rem}
\label{storyrem}
We now elaborate on the path that led to the proof that $|C(\Q)|=6$
(although it isn't actually used in our proof).
The subgroup of~$J(\Q)$ generated by differences of known rational points
on~$C$ has rank~$4$.
Since $J$ is $\Q$-isogenous to $\Jac(D)^2$, the rank of $J(\Q)$ is even, so it must 
be either $4$ or $6$.
The $S_3$-action shows that $|C(\Q)|$ is divisible by $6$.
A Chabauty argument at the prime $2$, using \cite{StollChabauty}, 
then gives that $|C(\Q)|$ is $6$ or $12$.
The argument is as follows.
Consider the pairing
$J(\Q_2) \times \Omega(C/\Z_2) \to \Q_2$
defined by $G, \omega \mapsto  (G,\omega) := \int_G\omega$
where $\Omega(C/L)$ is the set of holomorphic differentials
on $C$ over $L$ and $G$ is a degree 0 divisor on $C$.
View $Z \subset C(\Q) \subset J(\Q)$ (fixing a basepoint in $Z$) and let
$$
V := \{ \omega\in \Omega(C/\Z_2) : (J(\Q),\omega) = 0 \} \subseteq
V_0 := \{ \omega\in \Omega(C/\Z_2) : (Z,\omega) = 0 \}.
$$
Let
$\widetilde{V}$ (resp., $\widetilde{V_0}$) be the image of $V$ 
(resp., $V_0$) under the reduction map
$\Omega(C/\Z_2) \to \Omega(C/\F_2)$.
Then
$\dim_{\F_2}\widetilde{V} = \rk_{\Z_2}V 
\ge 12 - \rk\,{J(\Q)} \ge 6$.
Let 
$$
W := \{ \omega\in \Omega(C/\F_2) : \ord_P(\omega) \ge 2 \\
\text{ for all $P\in \{ P_0, P_1, P_\infty \}$} \}.
$$
If $|C(\Q)| = 12$, each fiber of $C(\Q) \to C(\F_2)$ has size 4, and
it follows (using \cite{StollChabauty}) that
$\widetilde{V} \subseteq W \cap \widetilde{V_0}$.
So if $\dim(W \cap \widetilde{V_0}) < 6$,
then $|C(\Q)|=6$. 
Balakrishnan, using ideas of Kedlaya and Wetherell,
computed $\int_P^Q\omega$ for $\omega\in \Omega(C/\Z_2)$ and $P, Q \in Z$.
This gives $V_0$ and $\widetilde{V_0}$.
Unfortunately, $\dim(W \cap \widetilde{V_0}) = 6$.
However, if $\rk\,{J(\Q)}$ were $4$,
then
$\dim_{\F_2}\widetilde{V} = \rk_{\Z_2}V 
\ge 12 - \rk\,{J(\Q)} =  8$.
This would imply that $|C(\Q)|=6$,
since  $\widetilde{V} \subseteq W \cap \widetilde{V_0}$
when $|C(\Q)| = 12$.
As explained in Appendix A, we had reason to believe that $\rk_\Z J(\Q) = 6$.
This gave motivation to find additional generators of $J(\Q)$, which we do 
in the proof of the following theorem.
\end{rem}

\begin{thm}
\label{rkQ}
$\rk_\Z J(\Q) = 6$.
\end{thm}

\begin{proof}
By Lemma \ref{rkQlem}, we have $\rk_\Z J(\Q) \le 6$.

We look for closed
points of higher degree on the genus $6$ quotient $D$ of $C$ defined in Remark \ref{Ddefrem}. 
For $D$, we use the model in~$\BP^3$ obtained
as the image of the map sending a point $(v, w)$ on~$C$ to
$(v + \frac{1}{v} : w + \frac{1}{w} : \frac{v}{w} + \frac{w}{v} : 1)$.
This model is a smooth curve of genus~6 (and degree~10) still having bad reduction
only at~$7$. We intersect $D$ with hyperplanes and quadrics of increasing
height and split the resulting divisor as a sum of prime divisors. If it splits,
we check if the new prime divisor leads to a larger subgroup in the Jacobian
of~$D$. For this check, we use the homomorphism
\[ \Pic(D) \longrightarrow \prod_{p \in S} \Pic(D/\F_p) \]
where $S$ is a suitable set of primes (we used $S = \{2,3,5,11,13,17\}$);
compare~\cite{StollXdyn06} for details. 

In this way, we find a point of degree~4 that leads to a larger group that is 
still of rank~$2$, and a point of degree~8 that increases the rank to~$3$.
Pulling back these points to~$C$, we obtain a point of degree~$8$ of the form
\[  (-\alpha^7 + 2 \alpha^6 + \alpha^4 + \alpha^2 + 2, \alpha) \]
with $\alpha$ a root of
\[ t^8 - 2 t^7 - t^5 - t^3 - 2 t + 1 \,, \]
and a point of degree~$16$ of the form
\begin{align*}
  \Bigl(\tfrac{1}{216}(187 &\beta^{15} + 894 \beta^{14} + 98 \beta^{13} + 408 \beta^{12}
                   + 1037 \beta^{11} + 1245 \beta^{10} + 1754 \beta^9 + 1371 \beta^8 \\
                        &{} + 783 \beta^7 + 1702 \beta^6 + 1497 \beta^5 + 793 \beta^4
                            + 708 \beta^3 + 250 \beta^2 + 450 \beta + 575),
                         \beta\Bigr)
\end{align*}
with $\beta$ a root of
\begin{multline*}
 t^{16} + 5 t^{15} + 2 t^{14} + 4 t^{13} + 5 t^{12} + 10 t^{11} + 11 t^{10}
            + 13 t^9 + 6 t^8 \\
        + 13 t^7 + 11 t^6 + 10 t^5 + 5 t^4 + 4 t^3 + 2 t^2 + 5 t + 1 \,.
\end{multline*}
Denote the corresponding prime divisors on~$C$ by $Q_1$ and~$Q_2$, respectively. 
Let $Q_3=\sigma(Q_1)$ and~$Q_4=\sigma(Q_2)$, where $\sigma\in\Sigma$
is the map $(v,w) \mapsto (\frac{1}{1-v},w)$.  
Let $H \subset J(\Q)$ be the subgroup generated by degree-zero 
linear combinations of the ten prime divisors:
\begin{equation}
\label{tpd}
\{P_0, P_1, P_\infty, P_2, P_3, P_4, Q_1, Q_2, Q_3, Q_4\}.
\end{equation}

We compute that $|J(\F_{13})| = 3^6\cdot 7^4\cdot 13^2\cdot 349^2$, and therefore 
$J(\Q)$ has no $5$-torsion.  (By considering additional reductions 
one can show that $|J(\Q)_\tors|$ divides $7^2$.)
In addition, we check that the image of 
$$
H \too J(\F_2) \times J(\F_3) \times J(\F_{11})
$$
has a quotient isomorphic to $(\Z/5\Z)^6$.  It follows that $\rk_\Z H \ge 6$, 
so $\rk_\Z J(\Q) = 6$ and $H$ has finite index in $J(\Q)$.
\end{proof}

\begin{rem}
In the course of the calculations in the proof of Theorem \ref{rkQ}, we showed also 
that the subgroup $H \subset J(\Q)$ is free of rank $6$, and that the relations 
among the divisors \eqref{tpd} are generated by:
\begin{align*}
  P_0 + P_1 + P_\infty &\sim P_2 + P_3 + P_4, \\
  P_1 + P_4 + Q_1 &\sim P_2 + P_\infty + Q_3, \\
  5 P_0 + 3 P_1 + 4 P_\infty + 4 P_2 + 5 P_3 + 3 P_4 &\sim 2 Q_1 + Q_3.
\end{align*}
\end{rem}

\section{Proof of Theorem \ref{final}}
\label{finalpfsect}

Since the Mordell-Weil rank of $J(\Q)$ is smaller than the genus of $C$, we can apply
Chabauty's method (see~\cite{Chabauty,Coleman,StollChabauty})
to calculate $C(\Q)$.

Define divisors on $C$:
\begin{gather*}
D_x = \sum_{\zeta \in \bmu_7} (x,\zeta) ~ \text{for $x \in \{0,1,\infty\}$}, \quad 
G_1 = \sum_{i=1}^3(\alpha_i,0), \quad  G_2 = \sum_{i=4}^6(\alpha_i,\infty),
\end{gather*}
where (as before), $\alpha_1, \alpha_2, \alpha_3$ (resp., $\alpha_4, \alpha_5, \alpha_6$) 
are the roots of $v^3-2v^2-v+1$ (resp., $v^3-v^2-2v+1$).  
One checks easily that the divisors of the rational functions 
$v$, $w$, and $v^3-v^2-2v+1$ are given by
\begin{equation}
\label{pre7.1}
(v) = D_0 - D_\infty, \quad (w) = G_1 - G_2, \quad (v^3-v^2-2v+1) = 7G_2 - 3D_\infty.
\end{equation}

If $L$ is a field of characteristic different from $7$, 
let $\Omega(C/L)$ denote the $L$-vector space of holomorphic differentials 
on $C/L$.  If $\omega \in \Omega(C/L)$, let $(\omega)$ denote the divisor of $\omega$. 

\begin{lem}
\label{pt}
Suppose $L$ is a field of characteristic different from $7$.  
\begin{enumerate}
\item
The divisor of the differential $dv$ is $(dv) = 6G_1 + 6G_2 - 2D_\infty$.
\item
A basis for the $L$-vector space $\Omega(C/L)$ is given by
$$
\omega_{i,j} := \frac{v^i w^j}{(v^3-v^2-2v+1)w^6} dv, \quad 0 \le i \le 1,\; 0 \le j \le 5.
$$
\end{enumerate}
\end{lem}

\begin{proof}
The function $v$ has (simple) poles 
at each of the $7$ points $\{(\infty,\zeta) : \zeta \in \bmu_7\}$, and no other poles.  Hence 
$\ord_{(\infty,\zeta)}(dv) = -2$, and $\ord_P(dv) \ge 0$ for all other points $P$.  
If $\alpha$ is a root of $v^3-2v^2-v+1$, then the equation 
for $C$ shows that $\ord_{(\alpha,0)}(v-\alpha)$ is a (positive) multiple of $7$. 
Since the polar divisor of $v$ is $D_\infty$, we conclude that $\ord_{(\alpha,0)}(v-\alpha) = 7$, 
and 
$$
\ord_{(\alpha,0)}(dv) = \ord_{(\alpha,0)}(d(v-\alpha)) = 6.
$$
Similarly, if $\alpha$ is a root of $v^3-v^2-2v+1$ then $\ord_{(\alpha,\infty)}(dv) = 6$.  
Since the divisor $(dv)$ has degree $2g-2 = 22$, we conclude that 
$(dv) = 6G_1+6G_2-2D_\infty$, giving (i).

It now follows from \eqref{pre7.1} that the differential $\omega_{i,j}$ has divisor
$$
(\omega_{i,j}) = i D_0 + (1-i) D_\infty + jG_1 + (5-j)G_2.
$$
In particular $\omega_{ij}$ is holomorphic if (and only if) $0 \le i \le 1$ and 
$0 \le j \le 5$.

Since $v$ is transcendental over $L$ and $w$ has degree $7$ over $L(v)$, the set 
$\{v^i w^j : 0 \le i \le 1, 0 \le j \le 5\}$ is linearly independent over $L$, 
so $\{\omega_{ij} : 0 \le i \le 1, 0 \le j \le 5 \}$ is linearly independent over $L$.
Since $\dim_L(\Omega(C/L))= \genus(C) = 12$, we have (ii).
\end{proof}

Let $\Omega(C/\Z_5)$ be the $\Z_5$-span of the differentials $\omega_{i,j}$ with
$0 \le i \le 1$ and $0 \le j \le 5$.
Consider the bilinear pairing 
$$
J(\Q_5)/J(\Q_5)_\tors \times \Omega(C/\Z_5) \to \Q_5
$$
of free $\Z_5$-modules of rank $12$ with trivial left and right kernel
that is used on p.~1210 of \cite{StollChabauty}.
Let $V \subset \Omega(C/\Z_5)$ be the orthogonal complement under this pairing of 
(the closure of) $J(\Q) \subset J(\Q_5)$.  
By Theorem \ref{rkQ} we have $\rk_\Z J(\Q) = 6$.

Let $\tilde{V} \subset \Omega(C/\F_5)$ be the image of $V$ 
under the (surjective) reduction map $\red_5: \Omega(C/\Z_5) \to \Omega(C/\F_5)$.  
Since $\rk_{\Z_5}(\Omega(C/\Z_5)) = 12 = \dim_{\F_5}(\Omega(C/\F_5))$,
we have $\ker(\red_5) = 5\Omega(C/\Z_5)$.
Since $\Omega(C/\Z_5)/V$ is torsion-free, we have 
$5V= V \cap 5\Omega(C/\Z_5)=\ker(\red_5|_V)$.
Thus $\tilde{V} \cong V/5V$, so  $\dim_{\F_5}(\tilde{V}) = \rk_{\Z_5}(V) = 6$.

We will show that for each point $P \in C(\F_5)$, there exists a 
differential $\omega_P \in \tilde{V}$ that does not vanish at~$P$.
Proposition 6.3 of \cite{StollChabauty} then shows that there is at most
one point in $C(\Q)$ that reduces to $P$.
Since the set $Z\subseteq C(\Q)$ bijects onto~$C(\F_5)$ via
the reduction map, that will show $Z= C(\Q)$, as desired.

Next we determine the space $\tilde{V}$ explicitly.

Using lattice basis reduction (with higher weights on the higher-degree
points), we find the following generators of the intersection of the known
finite-index subgroup of~$J(\Q)$ and the kernel of reduction mod~$5$:
\begin{align*}
  B_1 &= P_1 - P_\infty + P_2 - P_4 - 6 Q_2  + 6 Q_4, \\
  B_2 &= -4 P_0 + 9 P_\infty+ 5 P_2 - 3 P_3 + P_4  + 4 Q_1 - 7 Q_3 + Q_4, \\
  B_3 &= -4 P_0 + 9 P_1 + P_2 - 3 P_3 + 5 P_4 - 7 Q_1 + Q_2 + 4 Q_3, \\
  B_4 &= -2 P_0 - 2 P_1 + 4 P_\infty - 11 P_2 + 13 P_3 - 2 P_4 - 4 Q_1 + Q_2 - 2 Q_3 + 2 Q_4, 
         \\
  B_5 &=  - 6 P_1 + 6 P_\infty - 9 P_2 + 9 P_4 - 2 Q_1 - Q_2 + 2 Q_3 + Q_4, \\
  B_6 &= 10 P_0- 8 P_1 - 14 P_\infty + P_2 - 7 P_3  - 6 P_4
          - 12 Q_1 + 11 Q_2 - 15 Q_3 + 4 Q_4.
\end{align*}
For each of these divisors~$B_m$, we compute the Riemann-Roch space of
$\pm B_m + 12 P_0$ (with sign chosen in an attempt to make the computation
more efficient). We check that it is of dimension one and that the same
is true for the corresponding Riemann-Roch space over~$\F_5$. Let
$f_m$ be a function spanning the space; then it follows that the divisor
of~$f_m$ must be $\mp B_m - 12 P_0 + D_m$ with an effective divisor~$D_m$
of degree~$12$ supported on points having the same reduction mod~$5$ as~$P_0$.
(Note that the divisor of the reduction of~$f_m$ mod~$5$ must be the reduction
of~$\mp B_m$, since this reduction is a principal divisor and the Riemann-Roch
space of the reduction of $\pm B_m + 12 P_0$ has dimension one.)

We write each basis element $\omega_{ij}\in\Omega(C/\Z_5)$ 
as a power series in~$t$ times~$dt$,
where $t = v$ is a uniformizer at~$P_0$ (and also a uniformizer at the reduction
of~$P_0$ mod~$5$). We then integrate formally to obtain the corresponding
logarithms $\lambda_{ij}$ as power series in~$t$. These power series converge
5-adically on points in the same residue class as~$P_0$. The expansion of~$f_m$
as a Laurent series in~$t$ is of the form $t^n F_m(t)$ where $n$ is
$-12 \mp{}\!\!$ the multiplicity of~$P_0$ in~$B_m$. After scaling
by a power of~$5$ if necessary, the series $F_m(t)$ is in~$\Z_5[\![t]\!]$, the coefficients
of~$t^k$ with $k < 12$ have positive valuation, and the coefficient of~$t^{12}$
is a 5-adic unit. This reflects the facts that $D_m$ has the same reduction
mod~$5$ as~$12 P_0$, and $P_0$ is not in the support of the reductions mod~$5$ 
of any of the $Q_i$ or $P_j$ with $j \neq 0$ (as is easily checked).
We multiply this series by an invertible series in~$\Z_5[\![t]\!]$
such that we obtain a polynomial of degree~$12$ (to sufficient 5-adic precision).
The roots of this polynomial are then the values~$t_1, \ldots, t_{12}$
of~$t$ at the points in the
support of~$D_m$. We can then easily compute the power sums of these
values (by looking at the logarithmic derivative of the reverse of the
polynomial, see again~\cite{StollXdyn06}) and therefore evaluate the integrals
\[ \int_{[D_m - 12 P_0]} \omega_{ij} = \sum_{k=1}^{12} \lambda_{ij}(t_k) \,. \]
The values are in~$5\Z_5$. Dividing by~$5$ and reducing mod~$5$, we obtain
a linear relation that every differential in $\tilde{V}$ has to satisfy.

We used Magma~\cite{MAGMA} for these computations. A recent version (2.17-1)
spends less than 16~hours on current hardware to produce the relations we
are looking for. We obtain a 6-dimensional space of relations generated
by the following elements (in terms of the basis dual to
$\omega_{00}, \ldots, \omega_{05}, \omega_{10}, \ldots, \omega_{15}$):
\begin{align*}
  &(1, 0, 0, 0, 0, 0, 4, 1, 1, 3, 3, 4), \\
  &(0, 1, 0, 0, 0, 0, 1, 1, 0, 2, 3, 0), \\
  &(0, 0, 1, 0, 0, 0, 3, 4, 0, 0, 4, 2), \\
  &(0, 0, 0, 1, 0, 0, 3, 1, 0, 1, 1, 2), \\
  &(0, 0, 0, 0, 1, 0, 0, 2, 3, 0, 0, 4), \\
  &(0, 0, 0, 0, 0, 1, 1, 2, 2, 4, 4, 2).
\end{align*}
As stated above, Theorem \ref{final}
would follow if we can find, for $P \in C(\F_5)$, a 
differential $\omega_P \in \tilde{V}$ that does not vanish at~$P$.
Next we give an explicit differential $\omega$ that works for all $P \in C(\F_5)$.
Let $h(v,w)=3+w+3w^3+2w^4+v(w^2+2w^3)$, $f_2(v)=v^3-v^2-2v+1$, and
$$
\omega := 3\omega_{00}+\omega_{01}+3\omega_{03}+2\omega_{04}+\omega_{12}+2\omega_{13} 
   = \frac{h(v,w)}{w^6}\cdot\frac{dv}{f_2(v)} \in \Omega(C/\F_5).
$$
The coefficients of the $\omega_{ij}$ that define $\omega$ satisfy the relations above, so $\omega \in \tilde{V}$.

If $0 \le i \le 4$ then \eqref{pre7.1} and Lemma \ref{pt}(i) show that 
$dv/f_2(v)$ does not vanish at $P_i$.
By inspection $h(v,w)/w^6$ does not vanish at $P_i$, 
so $\omega$ does not vanish at $P_i$.

Similarly, $dv/f_2(v)$ vanishes at $P_\infty$ 
but $v\,dv/f_2(v)$ does not,
by \eqref{pre7.1} and Lemma \ref{pt}(i).  
Since $w^2 + 2w^3$ also does not vanish at $P_\infty$, 
it follows that $\omega$ does not vanish at $P_\infty$. 
This concludes the proof of Theorem \ref{final}. 

As an independent check, we have performed a similar computation on the curve~$D$.
There are four points in~$D(\F_5)$ that are images of points in~$C(\F_5)$. For
each of these points, we find a differential that kills the Mordell-Weil group
of~$D$ and whose reduction mod~$5$ does not vanish at the given point. This
shows that $D$ can have at most four rational points that are in the image
of~$C(\Q)$; this number is accounted for by the known points.


\begin{appendix}

\section{Conjectural determination of the rank}

Recall that $J$ is isogenous to the square of the Jacobian~$J_D$ of a curve~$D$
over~$\Q$ of genus~6 defined in Remark \ref{Ddefrem}.
In this appendix, we explain the computations that led us to believe that
the Mordell-Weil rank of~$J_D$ is~3 and hence that the
rank of~$J(\Q)$ is~6, assuming standard conjectures on $L$-series and the
conjecture of Birch and Swinnerton-Dyer for~$J_D$.

Since we already knew
that the rank of~$J_D(\Q)$ must be either 2 or~3, it would be sufficient to verify
that a positive sign in the functional equation for the $L$-series of~$J_D$
is not consistent with standard conjectures. The result would be even more
convincing if it could also be shown that a negative sign, and 
analytic rank~3 in particular, 
\emph{is} consistent with the conjectures. 

Two of the
main ingredients for the computation are the conductor of $J_D$ and the 
Fourier coefficients of its $L$-series.  
Since $D$ (and therefore~$J_D$) has bad reduction only at~7, the conductor is of the form 
$7^n$ for some $n$.  Since the reduction of~$J_D$ at~7 is totally
unipotent (this is shown by computing a proper regular model of~$D$ over~$\Z_7$,
whose special fiber turns out to be tree-like and without components of positive
genus), it follows that $n \ge 2\cdot\genus(D) = 12$.  An upper bound  $n \le 26$ 
follows from~\cite{LRS} or \cite{BrumerKramer}.

The Fourier coefficients of the
$L$-series can be obtained by counting points on~$D$ or~$C$
over all finite fields~$\F_{p^e}$
for $p^e$ below some bound~$N$ (and $e \le \genus(D) = 6$, making use of the Weil
conjectures). This provides the first~$N$ coefficients. In our case, Balakrishnan,
Sutherland, and Kedlaya provided these coefficients for $N = 10^7$ and $p \neq 7$.
The Euler factor at~7 is trivial, since
the reduction is totally unipotent. This many coefficients turned out to
be enough to produce satisfactory results (``a handful'' of digits of accuracy,
according to Rubinstein).

Rubinstein performed the $L$-series computations for us,
using his \texttt{lcalc} package~\cite{lcalc}. He first checked 
that the rank of~$J_D(\Q)$ being~2 is not consistent with the standard
conjectures, by computing the central $L$-value assuming the root number to
be~$+1$ and the conductor to be $7^{12}, 7^{13}, \ldots, 7^{26}$. In each
case the result is clearly non-zero.  
Since we know that the rank is $2$ or $3$, the Birch and Swinnerton-Dyer conjecture 
leads to the conclusion that the rank must be $3$.  This computation is based
on the approximate functional equation; see~\cite[Thm.~1]{Rubinstein}.

As a further check, Rubinstein used the approximate functional equation to
compute the first 17 zeros on the critical line (normalized to be $s = 1/2$) 
with positive imaginary part, assuming the root number to be $-1$ and the 
conductor to be $7^{26}$.  He then compared the two sides of the ``explicit
formula'', assuming a triple zero at the critical point and using the function
$\phi(x) = \left(\frac{\sin 4x}{4x}\right)^4$ (whose Fourier transform has
compact support of a size that allows essentially exact computation of the
right hand side of the explicit formula with the known Fourier coefficients). 
The values obtained
for $\sum_{\gamma} \phi(\gamma - t)$, where $\gamma$ runs through the zeros
and $0 \le t \le 3$, agree almost perfectly.  See Figure~\ref{plot} for the corresponding 
calculation using only $16$ zeros away from the real axis, where the 
two graphs can be seen to begin to diverge for $t > 2.8$.

\begin{figure}[htb]

\vspace{-5mm}

\begin{center}
\includegraphics[width=\textwidth]{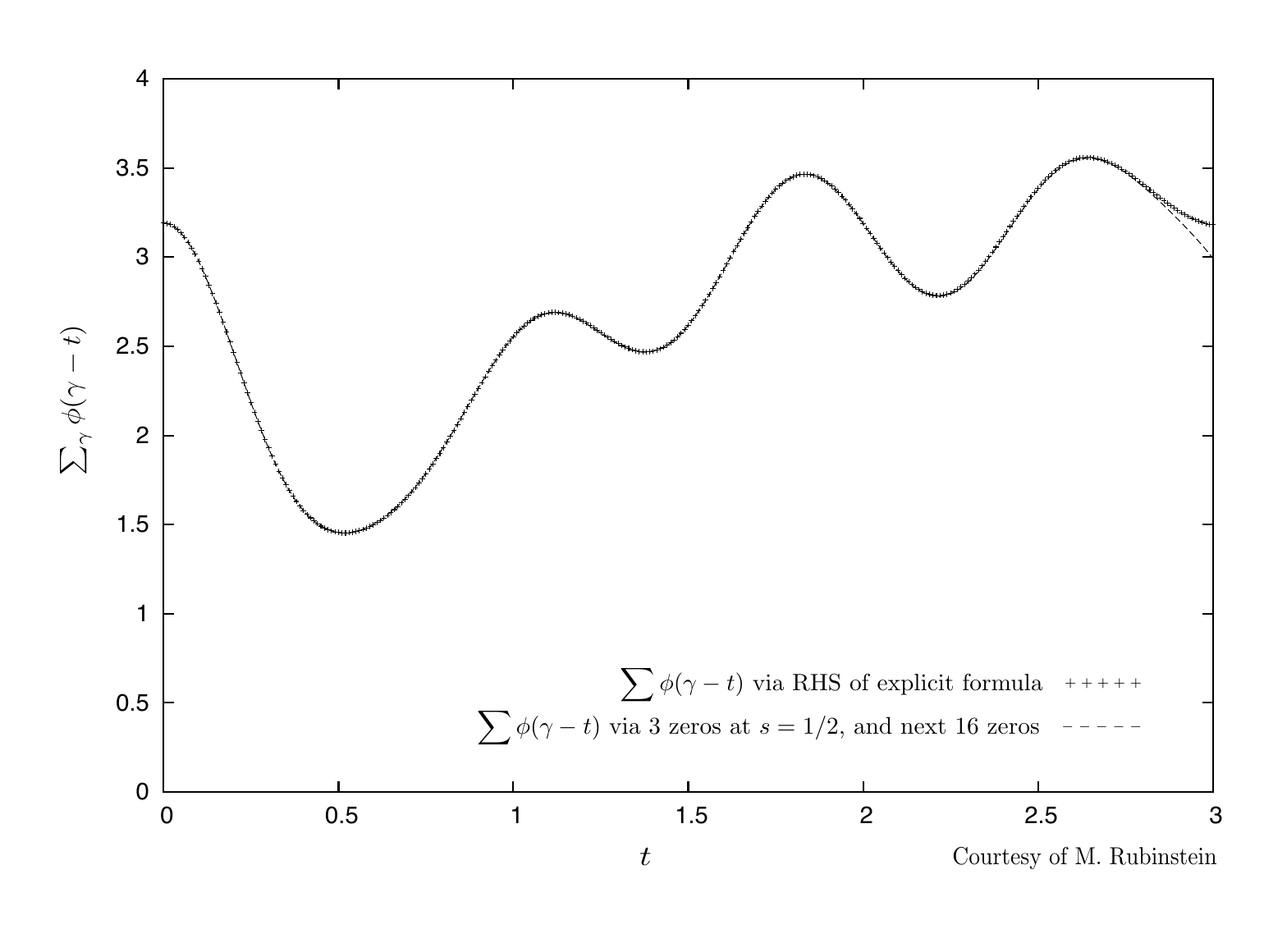}
\end{center}

\vspace{-10mm}
 
\caption{\label{plot}Comparison of left and right sides of the explicit formula.}
\end{figure}

Heuristic considerations indicate that the discriminant of $C$ should be 
$7^{52}$. The given model of $C$ is already regular at 7, so the exponent of 
the conductor should equal that of the discriminant. Since the conductor of 
$C$ is the square of that of $D$, this provides further evidence that $D$ has 
conductor $7^{26}$.

Note that the global root number of the $L$-function of $J_D$
should be the product of local root numbers,
which in our case will all be~$+1$ except at $p = 7$. So it would be sufficient
to determine the local root number at~7 in order to get the parity of the rank.
However, computing such a local root number when the reduction is unipotent 
seems to be rather hard.

\end{appendix}


\end{document}